\documentclass[12pt]{amsart}
\usepackage{amsmath, amsthm, amssymb}
\usepackage{amsfonts}
\usepackage{fullpage}
\usepackage{color}
\usepackage{hyperref}
\usepackage{enumitem}
\usepackage{bm}
\usepackage{verbatim}

\usepackage[dvipsnames]{xcolor}
\usepackage{graphicx}  
\DeclareGraphicsExtensions{.pstex,.eps}

\usepackage{xcolor}

\renewcommand{\ell}{{l}}  

\newcommand{\R}{{\mathbb{R}}}
\newcommand{\Z}{{\mathbb Z}}
\newcommand{\C}{{\mathbb C}}

\renewcommand{\P}{{\mathbb P}}
\renewcommand{\S}{{\mathbb S}}

\newcommand{\N}{{\mathbb N}}

\newcommand{\EE}{{\mathcal E}}

\newcommand{\NN}{{\mathcal N}}
\newcommand{\FF}{{\mathcal F}}

\newcommand{\GG}{{\mathcal G}}

\renewcommand{\Re}{\mathop{\rm Re}\nolimits}
\renewcommand{\Im}{\mathop{\rm Im}\nolimits}

\renewcommand{\div}{\operatorname{div}}

\theoremstyle{plain}

\newtheorem{thm}{Theorem}[section]
\newtheorem{prop}[thm]{Proposition}
\newtheorem{cor}[thm]{Corollary}
\newtheorem{lemma}[thm]{Lemma}

\theoremstyle{definition}

\newtheorem{rem}{Remark}[thm]

\numberwithin{equation}{section}

\def\squarebox#1{\hbox to #1{\hfill\vbox to #1{\vfill}}}

\newcommand{\<}{\langle}
\renewcommand{\>}{\rangle}

\renewcommand{\d}{\partial}
\newcommand{\ep}{\epsilon}
\newcommand{\lV}{\lVert}
\newcommand{\rV}{\rVert}

\def\Ga{\Gamma}
\def\de{\delta}
\def\De{\Delta}
\def\ep{\epsilon}

\def\la{\lambda}
\def\La{\Lambda}

\def\Om{\Omega}

\def\nab{\nabla}

\def\al{\alpha}
\def\les{\lesssim}
\def\c{\cdot}

\title[Incompressible Schr\"odinger flow]
{Asymptotic behaviors of incompressible Schr\"{o}dinger flow for small data in three dimensions}

\author[J. Huang]
{Jiaxi Huang}

\author[L. Zhao]
{Lifeng Zhao}

\address{School of Mathematics and Statistics, Beijing Institute of Technology, Beijing 100081, P.R. China}
\email{jiaxih@bit.edu.cn}

\address{School of Mathematical Sciences, University of Science and Technology of China, Hefei, Anhui,
230026, PR China}
\email{zhaolf@ustc.edu.cn}



\begin{document}

\begin{abstract}
The incompressible Schr\"odinger flow is a Madelung’s
hydrodynamical form of quantum mechanics, which can simulate classical fluids with particular advantage in its simplicity and its ability of capturing thin vortex dynamics. This model enables robust simulation of intricate phenomena such as vortical wakes and interacting vortex filaments.

In this article, we prove the global regularity and asymptotic behaviors for incompressible Schr\"odinger flow with small and localized data in three dimensions. We choose a suitable gauge to rewrite the system, and then use Fourier analysis and vector fields method to prove global existence and asymptotic behaviors.
\end{abstract}

\date{\today}
\maketitle


\setcounter{tocdepth}{1}

\section{Introduction}
In this paper, we consider the incompressible Sch\"{o}dinger flow initial-value problem
\begin{equation}           \label{Sys-Ori}
\left\{\begin{aligned}
&\d_t u+u\cdot\nab u+\nab P=\Delta u-\div(\nab\phi\odot\nab\phi),\\
&{\rm div}\  u=0,\\
&\d_t \phi+u\cdot\nab\phi=\phi\times\Delta\phi,\\
&(u,\phi)\big|_{t=0}=(u_0,\phi_0).
\end{aligned}\right.
\end{equation}
Here $d=2,3$ is the spatial dimensions, $u:[0,T]\times \R^d \rightarrow\R^d$ represents the velocity field of the flow, $P$ is the pressure function, and $\phi:[0,T]\times \R^d\rightarrow \S^2\subset\R^3$ denotes the magnetization field. The notation $\times$ is the cross product for vectors in $\R^3$, and the term $\nab\phi\odot\nab\phi$ denotes the $d\times d$ matrix whose $(i,j)$-th entry is given by $\d_i \phi\cdot\d_j \phi$ $(1\leq i,j\leq d)$, i.e.
\[
(\nabla \phi\odot\nabla \phi)_{ij}=  \sum_{k=1}^d  \partial_i \phi_k \partial_j \phi_k,
\]
and 
\[
(\div (\nabla \phi\odot\nabla \phi))_j=  \sum_{i=1}^d\sum_{k=1}^d  \partial_i(\partial_i \phi_k \partial_j \phi_k)=\sum_{k=1}^d \Delta \phi_k \partial_j\phi_k +\frac12 \partial_j |\nabla \phi_k|^2.
\]
The model \eqref{Sys-Ori} is a coupled system of the incompressible Navier-Stokes equations and Schr\"{o}dinger map flow, which can be used to describe the dispersive theory of magnetization of ferromagnets with quantum effects.

The incompressible Schr\"odinger flow \eqref{Sys-Ori} arises in the simulation of classical fluids in order to capture the coherent vortical structures and their dynamics. In \cite{Chern,Ch17,Ch16} they formulated and simulated classical fluids using a $\C^2$-valued Schrödinger equation subject to an incompressibility constraint. 
Such a fluid flow was called incompressible Schrödinger equation, which is used to simulated classical fluids with particular advantage in its simplicity and its ability of capturing thin vortex dynamics. 
Then they introduce the velocity Clebsch variable and vorticity Clebsch variable to derived the incompressible Schr\"odinger map flow \eqref{Sys-Ori}, which reveals a deep relation between Clebsch variables in hydrodynamics and spins in quantum mechanics.

The other motivation is that the incompressible Schr\"odinger flow (\ref{Sys-Ori}) can be seen as a special case of magnetoelasticity (see \cite{Br}). The model of magnetoelasticity describes a class of phenomena on the interaction between elastic and magnetic effects, whose discovery dates back at least to the 19th century. Brown\cite{Br} established the first rigorous phenomenological theory of magnetoelasticity. Tiersten also presented an essentially equivalent theory for magnetoelastic solids in two papers \cite{T64,T65}. Both these works consider magnetically saturated media undergoing large deformations, phrased in the Eulerian coordinate system. However, it lacks dissipation mechanisms and does not include the theory of micromagnetics.
Hence, Forster \cite{Fo16} further derived models for magnetoelastic materials by utilizing variational principles in a continuum mechanical setting. Recently, Jiang-Liu-Luo \cite{JLL23} proved the local-in-time existence of the evolutionary model
for magnetoelasticity with finite initial energy by employing the nonlinear iterative approach
to deal with the constraint on values of the magnetization $|\phi|=1$. After reformulating the evolutionary model near a constant
equilibrium, they also proved the global well-posedness to the evolutionary model for magnetoelasticity with zero external magnetic field under small size of initial data.
In the current paper, the incompressible Schr\"odinger flow \eqref{Sys-Ori} is seen as a model for magnetoelasticity without elasticity. The model exhibits the phenomenon of the interaction between fluid and magnetic effects.

The model \eqref{Sys-Ori} is reduced to the well-known Schr\"{o}dinger (or Landau-Lifshitz) flow of maps from $\R^d$ into $\S^2$ if we set $u\equiv 0$, which is an important model known as the ferromagnetic chain system. The local
well-posedness theory of Schrödinger map flows was established by Sulem-Sulem-Bardos \cite{SSB}, Ding-Wang \cite{DW} and McGahagan \cite{M}. The global well-posedness theory was started by Chang-Shatah-Uhlenbeck \cite{CSU} and Nahmod-Stefanov-Uhlenbeck \cite{NSU}. And the $d=1$ case with general targets was studied by Rodnianski-Rubinstein-Staffilani \cite{RRS}. 
For $\S^2$ target, the first global well-posedness result for Schr\"odinger flows in critical Besov spaces in $d\geq 3$ was proved by Ionescu-Kenig
\cite{IK} and Bejenaru \cite{B} independently. This was later improved to global regularity for small data in the critical Sobolev spaces in dimensions $d\geq 4$ in \cite{BeIoKe} and in dimensions $d\geq 2$ in \cite{BIKT}.
However, the question of small data global well-posedness in critical Sobolev spaces for general compact K\"ahler targets was more complicated, which was raised by Tataru in the survey report \cite{KTV}. Recently, Li \cite{Li0,Li} solved this problem using a novel bootstrap-iteration scheme to reduce the gauged equation to an approximate constant curvature system in finite times of iteration. We can refer to \cite{KTV} for a more detailed review.

The Schr\"odinger flow in \eqref{Sys-Ori} is essentially a Schr\"odinger equation after reformulating, in which the dispersion methods play a crucial role. We can refer to Tao's book in \cite{Tbook} for the classical and effective tools. As an application of these methods, here we list some works that motivated us to study the incompressible Schr\"odinger flow. Miao-Murphy-Zheng in \cite{MMZ14} considered a class of defocusing energy-supercritical nonlinear Schrödinger equations in four space dimensions and obtain the global well-posedness and scattering. Miao-Wu-Xu in \cite{MWX11} proved the global well-posedness for the cubic nonlinear Schr\"odinger equation with derivative in $H^{1/2}(\R)$, where the I-method combined with a new resonant decomposition technique was used.

Different from the above standard Schr\"odinger map equation, the presence of material derivative $\d_t+u\cdot\nab$ in \eqref{Sys-Ori} makes the $\phi$-equation become a quadratic Schr\"odinger equation. There are many works devoted to small data global regularity for quadratic Schr\"odinger equations of the form
\begin{equation*}  
    i\d_t u+\De u=N(u,\bar u,\nab u,\nab \bar u),\quad (t,x)\in \R\times \R^d.
\end{equation*}
To shorten the discussion, we focus here on recent developments corresponding to nonlinearities involving one $\nab u$ or $\nab\bar u$. On one hand, the nonlinearity with null structure makes estimates easier. On the other hand, the loss of derivatives makes estimates more complicated.
In dimension 3, Hayashi-Miao-Naumkin \cite{HaMiNa99} and Hayashi-Naumkin \cite{HaNa12} were able to prove global existence and scattering for small data and for
any quadratic nonlinearity involving at least one derivative. 
In dimension 2, Delort \cite{De} proved global existence for a nonlinearity of the form $u\nab u$ or $\bar u\nab\bar u$ by the vector fields method, a normal form transform and microlocal analysis. Finally, Germain-Masmoudi-Shatah \cite{GeMaSh} used the space-time resonance to prove global regularity and scattering for a nonlinearity $N=Q(u,u)+Q(\bar u,\bar u)$, where $Q$ is like a derivative for low frequencies, and the identity for high frequencies.

\subsection{The main theorem} Our objective in this paper is to establish the global in time regularity and scattering  of \eqref{Sys-Ori} with small and localized data in three dimensions by gauge theory and vector fields method, which extend the local solution of \eqref{Sys-Ori} proved by the first author \cite{Hu20}.

Here we start with two conservation laws for solutions of the system \eqref{Sys-Ori}. One of the conservation law is its energy defined by
\[ E(u,\phi)=\frac{1}{2}\|u\|_{L^2}^2+\int_0^t \|\nab u\|_{L^2}^2\ ds+\frac{1}{2}\|\nab\phi\|_{L^2}^2. \]
And \eqref{Sys-Ori} has the scaling invariance property: 
\[u(t, x) \rightarrow \la u(\la^2t, \la x),\quad\phi(t,x)\rightarrow \phi(\la^2 t,\la x).\]
Thus this scaling would suggest the critical Sobolev space for $(u,\phi)$ to be $\dot H^{\frac{d}{2}-1}\times \dot H^{\frac{d}{2}}$, and $d = 2$ is the energy critical case. 
Also the mass is conserved in the \eqref{Sys-Ori}
\begin{equation}  \label{mass-Conservation}
M(\phi)=\frac{1}{2} \| \phi-Q\|_{L^2}^2,\qquad \text{if}\ \|\phi_0-Q\|_{L^2}<\infty,\text{ for some}\ Q\in\S^2. \end{equation}

To state our main theorem precisely we need some notation. We define the scaling vector-field $S$ and the rotation vector fields $\Om_i$ by
\begin{align*}
    S=t\d_t+x\cdot\nab,\quad (\Om_1, \Om_2, \Om_3)=(x_2\d_3-x_3\d_2, x_3\d_1-x_1\d_3, x_1\d_2-x_2\d_1).
\end{align*}
Then we denote the set of vector-fields as
\begin{align*}
    \mathcal Z=\big\{\d_t, \d_1, \d_2, \d_3, \Om_1, \Om_2, \Om_3,  S\big\},
\end{align*}
and the time independent analogue of $\mathcal Z$ as
\begin{align*}      
	\mathcal Z_0= \big\{ \d_1, \d_2, \d_3, \Om_1, \Om_2, \Om_3, S_0  \big\},\quad \tilde S_0=\tilde S-2t\d_t.
\end{align*}
Given a point $Q\in \S^2$, we define the extrinsic Sobolev space $H^k_Q$ by 
\begin{equation*}
	H^k_Q:=\{ \phi:\R^d\rightarrow\R^3 : |\phi(x)|=1 \text{ and }\phi-Q\in H^k \}.
\end{equation*}

Then, our main result is as follows:
\begin{thm}[Small data global regularity and scattering]\label{Main_thm-ori}
Let $d=3$ and a point $Q\in\S^2$. Assume that the localized initial data $(u_0,\phi_{0})\in H^3\times H^4_Q$ satisfies
	\begin{equation}            \label{MainAss_ini}
		\|(u_0,\phi_0)\|_{H^3\times H^4}+\sum_{\Ga_0\in \mathcal Z_0}\lV  (\Ga_0 u_0,\nab \Ga_0\phi_0)\rV_{H^1}\leq \ep_0.
	\end{equation}
	Then the incompressible Schr\"odinger flow \eqref{Sys-Ori} admits a unique global solution $(u,\phi)\in H^3\times H^4_Q$ in three dimensions satisfying the energy bounds
	\begin{equation}\label{energy-bound}
		\|(u,\phi)\|_{H^3\times H^4_Q}+ \lV\nab u\rV_{L^2([0,t]:H^3)}+\sum_{\Ga\in \mathcal Z} \Big\{\lV (\Ga u,\nab \Ga\phi)\rV_{H^1}+ \lV\nab \Ga u\rV_{L^2([0,t]:H^1)}\Big\}\les \ep_0.
	\end{equation}
	for any $t\in[0,\infty)$. 
	Moreover, solution $(u,\phi)$ converges to the constant map $(0,Q)$ in the sense that
	\begin{equation} \label{thm-decay}
	    \lim_{t\rightarrow \infty}\big(\| u(t)\|_{L^\infty}+ \| \phi(t)-Q\|_{L^\infty}\big) =0.
	\end{equation}
	Furthermore, there exists $v_\infty, w_\infty$ and $\Psi_\infty\in H^1$ such that 
	\begin{align}  \label{thm-scatter}
	\lim_{t\rightarrow\infty}\|\d_j\phi-v_\infty \Re( e^{it\De}\Psi_{\infty,j})-w_\infty \Im( e^{it\De}\Psi_{\infty,j})\|_{H^1}=0.
	\end{align}
\end{thm}
\begin{rem}
	Since the map $\phi$ satisfies a quadratic Schr\"odinger equation, here we choose suitable vector fields to define the function spaces and Klainerman's generalized energy. The vector fields and function spaces are given in Section \ref{3.2} in detail.
\end{rem}

\subsection{Main ideas} \label{main-idea}
The main strategy to prove global regularity for \eqref{Sys-Ori} relies on an interplay between the control of high order energies and decay estimates, which is based on the Fourier analysis and vector-field method. The main ingredients include decay estimates, energy and weighted energy estimates, and $L^2$ weighted bounds on the profile $\Psi=e^{-it\De}\psi$ associated with differentiated fields $\psi$. However, there are still some difficulties to overcome.

\emph{1. The choice of vector fields.} In the prior works in \cite{GeMaSh,HaNa12} and so on for quadratic Schr\"odinger equations, the vector field $\mathcal G=x+2it\nab$ was applied to prove the small data global regularity, which commutes with the Schr\"odinger operator $i\d_t+\De$. 
However, the operator $\mathcal G$ does not commute with the heat operator $\d_t-\De$.
Furthermore, it works well only when the nonlinearities $N(u)$ is self-conjugate, namely, $N(e^{i\theta}u) = e^{i\theta}N(u)$ for all $\theta \in \R$. 
Thus these facts does not allow us to use the vector field $\GG$ to define the weighted energy functional. In view of the symmetries of Navier-Stokes equation and Schr\"odinger equation, as a system we would only choose the following suitable vector fields to define its weighted energy
\[ \{\d_t,\d_1,\d_2,\d_3,\tilde\Om_1,\tilde\Om_2,\tilde\Om_3,\tilde S\},  \]
where $\tilde \Om$, $\tilde S$ are perturbed angular momentum operators and perturbed scaling vector field, respectively.

\emph{2. Decay estimates of Schr\"odinger map.} However, the dispersive estimates could not connect the decay and weighted energy directly without the vector field $\GG=x+2it\nab$. Inspired by \cite{HaMiNa99} we can consider the operator $\GG\cdot\nab$ as a connection between decay estimates and weighted energy, and thus give the decay of differentiated field $\psi$. Precisely, the time decay estimates of the derivative of the solution $\psi$ are obtained by a priori estimate of the norm $\| \GG\nab\psi\|_{L^2}$. Then, we apply the inequality
\[ \| \GG\nab f\|_{L^2}\lesssim \| (x\cdot \nab+2it\De)f\|_{L^2}+\|\Om f\|_{L^2}, \]
which is valid for any smooth function $f$.
Note that the operator $x\cdot\nab +2it\De$ can be replaced by the scaling vector field $S=2t\d_t+x\cdot\nab$ via the identity $x\cdot\nab +2it\De=S+2it(i\d_t+\De)$.
Then it suffices to prove the decay estimates of the nonlinear term $(i\d_t+\De)\psi$ in $L^2$ for the solution $\psi$ of the nonlinear Schr\"odinger equation. This decay estimate is achieved by Fourier analysis and bootstrap assumptions, and hence the decay of the solution $\psi$ can be obtained.

\subsection{Notations}
The notation $A\lesssim B$ means there exists some universal constant $C>0$ such
that $A \leq  C B$. 
We denote $\mathcal R$ as the Riesz transformation $\mathcal R=\frac{\nab}{|\nab|}$. Let $\P$ be the Leray projection
\[\P=I_d+\nab(-\De)^{-1}\nab.\]

\subsection{Outline}
In section 2 we rewrite the $\phi$-equation for its differentiated fields. In section 3 we fix notation and state the main bootstrap proposition. We also state several lemmas. In section 4, we begin with the linear decay estimate, which is the key lemma in deriving the decay estimate of $\psi$. Then we use the bootstrap assumptions to derive various decay estimates of velocity $u$, connection coefficients $A$ and the nonlinearities in $\psi$-equation. In section 5, we use these decays and Fourier analysis to prove the energy and weighted energy estimates, and also give the scattering of $\nab\phi$ in Coulomb gauge. Finally, we use local existence and bootstrap Proposition \ref{Main_Prop} to prove Theorem \ref{Main_thm-ori}.

\bigskip 
\section{The differentiated equations}\label{sec2}
Instead of working directly on equation \eqref{Sys-Ori} for the map $\phi(t,x)$, it is convenient to study the equations satisfied by its derivatives $\d_m\phi$ for $m=1,\cdots,d+1$, where $\d_{d+1}=\d_t$. These are tangent vectors to the sphere at $\phi(t,x)$.

In this section we start with a smooth solution $\phi$ to the Schr\"odinger map equation and a smooth orthonormal frame $(v,w)$ in $T_\phi \S^2$. Then we construct the complex fields $\psi_m$ and the connection coefficients $A_m$, and derive the differentiated Schr\"odinger map equations satisfied by these functions under Coulomb gauge.

\begin{prop}[Coulomb gauge, \cite{BeIoKe}, Proposition 2.3]   \label{frame}
    Assume $T\in[0,1]$, $Q\in \S^2$, and 
    \begin{equation*}
        \left\{\begin{aligned}
        &\phi\in C([-T,T]:H^\infty_Q);\\
        &\d_t \phi\in  C([-T,T]:H^\infty).
        \end{aligned}\right.
    \end{equation*}
    Then there are continuous functions $v,\ w:\R^d\times [-T,T]\rightarrow\S^2$, $\phi\cdot v=0$, $w=\phi\times v$, such that 
    \begin{equation*}
        \d_m v,\ \d_m w\in C([-T,T]:H^\infty)\quad \text{for }m=1,\cdots,d,d+1
    \end{equation*}
    where $\d_{d+1}=\d_t$. In addition,
    \begin{equation}   \label{Coulomb}
        \text{if}\ A_m=(\d_m v)\cdot w\ \text{for }m=1,\cdots,d,\ \text{then }\sum_{j=1}^d \d_m A_m=0.
    \end{equation}
\end{prop}

Assume that $\phi,v,w$ are as in Proposition \ref{frame}. In addition to the connection coefficients $A_m$, we can define the differentiated variables $\psi_m$ for $m=1,\cdots,d,d+1$ and $A_{d+1}$
\begin{equation}   \label{diff-field}
\psi_m=\d_m\phi\cdot v+i\d_m\phi\cdot w,\quad A_{d+1}=(\d_t v)\cdot w.
\end{equation}
These allow us to express $\d_m\phi,\ \d_m v$ and $\d_m w$ in the frame $(\phi,v,w)$ as
\begin{equation}    \label{phi-v-w}
\left\{ \begin{aligned}
&\d_m \phi=v\Re \psi_m+w\Im \psi_m,\\
&\d_m v=-\phi\Re \psi_m+w A_m,\\
&\d_m w=-\phi\Im \psi_m- v A_m.
\end{aligned}
\right.
\end{equation}
Denote the covariant derivative as
\begin{equation*}
D_l=\d_l+iA_l.
\end{equation*}
We then obtain the curl type relations between the variables $\psi_m$
\begin{equation}     \label{com-Dpsi}
D_l\psi_m=D_m\psi_l
\end{equation}
and the curvature of the connection 
\begin{equation}     \label{cp-A}
D_lD_m-D_mD_l=i(\d_l A_m-\d_m A_l)=i\Im (\psi_l\Bar{\psi}_m).
\end{equation}

\medskip

Assume that the smooth function $\phi$ satisfies the Schr\"odinger map equation in \eqref{Sys-Ori}. Then we derive the Schr\"odinger equations for the function $\psi_m$.  
By \eqref{Sys-Ori}, $\phi\times v=w$ and $w\times \phi=v$, we have
\begin{align} \label{psi_d+1}
\psi_{d+1}=-u\c \psi+iD_l\psi_l.
\end{align}
Applying $D_m$ to \eqref{psi_d+1}, by the relations \eqref{com-Dpsi} and \eqref{cp-A} we obtain
\begin{align*}
D_{d+1}\psi_m=-\d_m u\c \psi-u\c D\psi_m+\Im(\psi_l\bar{\psi}_m)\psi_l+iD_lD_l\psi_m,
\end{align*}
which is equivalent to
\begin{equation}\label{Sch-eq}
    \begin{aligned}   
i(\d_t+u\c \nab)\psi_m+\De \psi_m=&\ -2iA \c\nab\psi_m+(A_{d+1}+(u+A)\c A-i\nab\c A)\psi_m\\
&\ -i\d_m u\c \psi+i\Im(\psi_l\bar{\psi}_m)\psi_l.
\end{aligned}
\end{equation}

Consider the system of equations which consists of \eqref{Sch-eq}, \eqref{com-Dpsi} and \eqref{cp-A}. The solution $\psi_m$ for the above system cann't be uniquely determined as it depends on the choice of the orthonormal frame $(v,w)$. Precisely, it is invariant with respect to the gauge transformation
\[
\psi_m\rightarrow e^{i\theta}\psi_m,\quad A_m\rightarrow A_m+\d_m \theta.
\]
In order to obtain a well-posed system one needs to make a choice which uniquely determines the gauge. Here we choose to use the Coulomb gauge \eqref{Coulomb}, which in view of \eqref{cp-A} leads to
\begin{equation}  \label{A_m-Eq}
\De A_m=\d_l \Im (\psi_l\bar{\psi}_m).
\end{equation}
Similarly, by \eqref{cp-A} and \eqref{psi_d+1} we also have the compatibility condition
\begin{align*}
\d_t A_m-\d_m A_{d+1}=\Im (\psi_{d+1}\bar{\psi}_m)=\Im [(-u\c \psi+iD_l\psi_l)\bar{\psi}_m].
\end{align*}
This combined with the Coulomb gauge implies
\begin{equation}\label{A_d+1-Eq}
\De  A_{d+1}=\d_m\Im [(u\c \psi-iD_l\psi_l)\bar{\psi}_m].
\end{equation}

In conclusion, under the Coulomb gauge $\nab\c A=0$ by \eqref{Sch-eq}, \eqref{A_m-Eq} and \eqref{A_d+1-Eq} we obtain the system for velocity $u$ and differentiated fields $\psi_m$
\begin{equation}\label{sys-2}
\left\{ \begin{aligned}
&\d_t u+u\cdot\nab u+\nab P=\Delta u-\d_j\Re(\psi \bar\psi_j),\\
&{\rm div}u=0,\\
&\begin{aligned}
i(\d_t+u\c \nab)\psi_m+\De \psi_m=&-2iA \c\nab\psi_m+(A_{d+1}+(u+A)\c A)\psi_m\\
&-i\d_m u\c \psi+i\Im(\psi_l\bar{\psi}_m)\psi_l,
\end{aligned}\\
&(u,\psi)\big|_{t=0}=(u_0,\psi_0),
\end{aligned}\right.
\end{equation}
where connection coefficients $A_m$ and $A_{d+1}$ are determined at fixed time in an elliptic fashion via the following equations
\begin{equation}   \label{A,Ad+1}
    \left\{\begin{aligned}
    &\De A_m=\d_l \Im (\psi_l\bar{\psi}_m),\\
    &\De  A_{d+1}=\d_m\Im [(u\c \psi-iD_l\psi_l)\bar{\psi}_m].
\end{aligned}   \right. 
\end{equation}
we can assume that the following
conditions hold at infinity in an averaged sense:
\[A_m(\infty) = 0;\quad A_{d+1}(\infty)=0.\]
These are needed to insure the unique solvability of the above elliptic equations in a suitable class of functions.

Therefore, to obtain our global solution for system \eqref{Sys-Ori}, it suffices to study the differentiated system \eqref{sys-2}-\eqref{A,Ad+1}.

\medskip

\bigskip 
\section{Function spaces and the main propositions} 
In this section, we start by summarizing our main definitions and notation. The decay estimates of heat operator and some useful lemmas are also stated here. Then we introduce the vector fields corresponding to the symmetry of \eqref{Sys-Ori}, which are used to define the function spaces $\mathcal H^k_\La$ and generalized energy $E_k$. Finally, we transfer the initial data from \eqref{Sys-Ori} to the differentiated system \eqref{sys-2}-\eqref{A,Ad+1}, thus we can state our main bootstrap proposition.

\subsection{Some analysis tools}
For a function $u(t,x)$ or $u(x)$, let $\hat{u}=\FF u$ denote the Fourier transform in the spatial variable $x$. Fix a smooth radial function $\varphi:\R^d \rightarrow [0,1] $ supported in $[-2,2]$ and equal to 1 in $[-1,1]$, and for any $j\in \Z$, let
\begin{equation*}
	\varphi_j(x):=\varphi(x/2^j)-\varphi(x/2^{j-1}).
\end{equation*}
We then have the spatial Littlewood-Paley decomposition,
\begin{equation*}
	\sum_{j=-\infty}^{\infty}P_j (D)=1, 
\end{equation*}
where $P_j$ localizes to frequency $2^j$ for $i\in \Z$, i.e,
\begin{equation*}
	 \FF(P_ju)=\varphi_j(\xi)\hat{u}(\xi).
\end{equation*}
For simplicity of notation, we set
\[
u_j=P_j u,\quad u_{\leq j}=\sum_{i=-\infty}^j P_i u,\quad u_{\geq j}=\sum_{i=j}^{\infty} P_i u.
\]

In the proof of energy estimates, one often needs to analyze the symbols. We define a class of symbol as follows
\begin{equation*}
\mathcal{S}^{\infty}:=\{m:\R^6\rightarrow\C,\ m\ \text{is continous and } \lV\mathcal{F}^{-1}m\rV_{L^1}<\infty\},
\end{equation*}
whose associated norms are defined as
\begin{equation*}
\lV m\rV_{\mathcal{S}^{\infty}}:=\lV \mathcal{F}^{-1}m\rV_{L^1},
\end{equation*}
and
\begin{equation*}
\lV m\rV_{\mathcal{S}^{\infty}_{k,k_1,k_2}}:=\lV m(\xi,\eta)\varphi_k(\xi)\varphi_{k_1}(\xi-\eta)\varphi_{k_2}(\eta)\rV_{\mathcal{S^{\infty}}}.
\end{equation*}
Then we have

\begin{lemma}[Bilinear estimate, \cite{IoPu}]      \label{Sym_Pre_lemma}
	Given $m\in\mathcal{S}^{\infty}$ and two well-defined functions $f_1,\ f_2$, then the following estimate holds:
	\begin{gather}  \label{Sym_Pre}
	\lV \mathcal{F}^{-1}\big(\int_{\R^3}m(\xi,\eta)\widehat{f}_1(\xi-\eta)\widehat{f}_2(\eta)d\eta\big)(x)\rV_{L^r}\lesssim\lV m\rV_{\mathcal{S}^{\infty}} \lV f_{L^1}\rV_{L^p} \lV f_2\rV_{L^q},\ \ \frac{1}{r}=\frac{1}{p}+\frac{1}{q}.
	\end{gather}
\end{lemma}

\subsection{Vector fields and function spaces}  \label{3.2}
Here we start with the usual Sobolev spaces $W^{s,p}$, which are defined by
\[ \|f\|_{W^{s,p}}=\sum_{k=0}^s\|f\|_{\dot{W}^{k,p}}=\sum_{k=0}^s\sum_{|\al|=k}\| \d^\al f  \|_{L^p},\quad \text{for }\al\in \N^d,\ s\in \N,\ 1\leq p<\infty.  \]
When $p=2$, the Sobolev spaces $H^s$ for any $s\in \R$ are also defined by
\[ \|f\|_{H^s}=\|\<\xi\>^s\widehat{f}(\xi)\|_{L^2},  \]
where $\<\xi\>=\sqrt{1+|\xi|^2}$ and $\widehat{f}$ is the Fourier transform of $f$.

Given a point $Q\in \S^2$, we recall the extrinsic Sobolev space $H^k_Q$ by 
\begin{equation*}
	H^k_Q:=\{ u:\R^d\rightarrow\R^3 : |u(x)|=1 \text{ and }u-Q\in H^k \},
\end{equation*}
which is equipped with the metric $d_Q(f,g)=\|f-g\|_{H^k}$. We also define the metric spaces 
\begin{equation*}
	H^\infty:=\bigcap_{k=1}^\infty H^k,\quad H^\infty_Q:=\bigcap_{k=1}^\infty H^k_Q.
\end{equation*}

Since $\phi$ actually satisfies a quadratic Schr\"odinger equation, we would use vector fields to define its weighted energy. Precisely, for the velocity $u(t,x)\in\R^3$, pressure $P(t,x)\in \R$ and direction field $\phi(t,x)\in \S^2$, we define the perturbed angular momentum operators $\tilde\Om$ as
\begin{equation*}   
	\tilde{\Om}_i u=\Om_i u+M_i\cdot u,\quad \tilde \Om_i P=\Om_i P,\quad \tilde \Om_i \phi=\Om_i \phi,
\end{equation*}
where $\Om=(\Om_1,\Om_2,\Om_3)=(x_2\d_3-x_3\d_2,x_3\d_1-x_1\d_3,x_1\d_2-x_2\d_1)=x\wedge \nab$ is the usual rotation vector field and the matrices $M_i$ are given by
\begin{equation*}
	M_1=\left(\begin{array}{ccc}
		0&0&0\\
		0&0&1\\
		0&-1&0
	\end{array}\right),\ \ 
	M_2=\left(\begin{array}{ccc}
		0&0&-1\\
		0&0&0\\
		1&0&0
	\end{array}\right),\ \ 
	M_3=\left(\begin{array}{ccc}
		0&1&0\\
		-1&0&0\\
		0&0&0
	\end{array}\right).
\end{equation*}
In view of the scaling invariance, we can also define the perturbed scaling vector-field $\tilde S$ applied to the unknowns $u,\ P,\ \phi$ in \eqref{Sys-Ori} by
\begin{align*}
	\tilde S u=(S+1)u,\quad \tilde S P=(S+2)P,\quad \tilde S\phi=S\phi,
\end{align*}
where $S=2t\d_t+x\cdot \nab$.

For simplicity of notations, we set
\begin{equation*}   
	Z_i\in \{\d_t,\d_1,\d_2,\d_3,\tilde\Om_1,\tilde\Om_2,\tilde\Om_3,\tilde S\},\quad Z^a=Z_1^{a_1}\cdots Z^{a_8}_8,
\end{equation*}
where $a=(a_1,a_2,\cdots,a_8)\in \N^8$. Note that the commutator of any two $Z$'s is again a $Z$.
Then we define the Klainerman's generalized energy $E_k$ for $k\leq 1$ by
\begin{align*}  
	\EE(t)&=\|(u,\phi)\|_{H^3\times H^4_Q}+ \lV\nab u\rV_{L^2([0,t]:H^3)}+\sum_{|a|=1} \Big\{\lV (Z^a u,\nab Z^a\phi)\rV_{H^1}+ \lV\nab Z^a u\rV_{L^2([0,t]:H^1)}\Big\}.
\end{align*}

In order to characterize the initial data, we introduce the time independent analogue of $Z$. Set
\begin{align*}     
	\Lambda\in \{ \d_1,\d_2,\d_3,\tilde\Om_1,\tilde\Om_2,\tilde\Om_3,\tilde S_0  \},\quad \tilde S_0=\tilde S-2t\d_t.
\end{align*}
The commutator of any two $\Lambda$'s is again a $\Lambda$. We define the function spaces $\mathcal H^k_\Lambda$ as  
\begin{equation*}         
	\mathcal H_\Lambda:=\Big\{ (u,\phi):  \|(u,\phi)\|_{H^3\times H^4_Q}+\sum_{|a|=1}\lV  (\Lambda^a u,\nab \Lambda^a\phi)\rV_{H^1}< \infty \Big\}.
\end{equation*}

Next, we state two basic lemmas.
\begin{lemma}[Decay estimates of heat operator]
	For any Schwartz function $f\in \mathcal{S}(\R^3)$, we have
	\begin{gather}        \label{Dec_heat}
	\lV e^{t\De}|\nab|^l f\rV_{\dot{W}^{N,q}}\lesssim t^{-\frac{3}{2}(\frac{1}{p}-\frac{1}{q})-\frac{l}{2}} \lV f\rV_{\dot{W}^{N,p}},\ \ \ \  1\leq p\leq q\leq \infty.
	\end{gather}
\end{lemma}
\begin{proof}
    The proof of \eqref{Dec_heat} is standard, which can be obtained by the estimates
    \begin{align*} 
    \| e^{t\De }|\nab|^l f \|_{L^q} &= \| \mathcal F^{-1}( e^{-t|\xi|^2}|\xi|^l ) \ast f \|_{L^q}\\
    &\les \| \mathcal F^{-1}( e^{-t|\xi|^2}|\xi|^l ) \|_{L^r} \| f \|_{L^p}\les t^{-\frac{3}{2}(\frac{1}{p}-\frac{1}{q})-\frac{l}{2}}\|f\|_{L^p},    
    \end{align*}
    where $1+1/q=1/r+1/p$.
\end{proof}

We also need the following lemma to connect the weighted $L^2$-norm of profile to weighted energy.
\begin{lemma}[Lemma 2.7, \cite{HJLZ21}]
Let $\Om=x\wedge \nab$ be rotation vector field. For any $N\geq 0$ and Schwartz function $f\in\mathcal{S}(\R^3)$, we have
	\begin{equation}           \label{xidxif}
	\lV\FF^{-1}(|\xi|\nab_{\xi}\widehat{f})\rV_{H^N}\lesssim \lV x\cdot \nab_x f\rV_{H^N}+\lV \Om f\rV_{H^N}+\lV f\rV_{H^N}.
	\end{equation}
\end{lemma}

\subsection{The main bootstrap proposition}
Here we consider the reformulated system \eqref{sys-2}-\eqref{A,Ad+1}, and state our main bootstrap proposition. The operators $\tilde \Om$ and $\tilde S$ applied to the variables $\psi=(\psi_1,\psi_2,\psi_3),\ A=(A_1,A_2,A_3)$ and $A_{d+1}$ are given by
\begin{align*}
	&\tilde \Om_i \psi=\Om_i \psi+M_i \cdot\psi,\quad \tilde \Om_i A=\Om_i A+M_i\cdot A,\quad \tilde \Om_i A_{d+1}=\Om_i A_{d+1}, \\
	&\tilde{S}\psi=(S+1)\psi,\quad \tilde S A_m=(S+1)A_m,\quad \tilde S A_{d+1}=(S+2)A_{d+1}.
\end{align*}
Then applying the vector fields $Z$ to \eqref{sys-2} and \eqref{A,Ad+1}, we obtain 
\begin{equation}    \label{sys-VF}
\left\{\begin{aligned}
&(\d_t-\De)Z^a u=-\sum_{b+c=a}C_a^b Z^b u\c \nab Z^c u-\nab Z^a P-\sum_{b+c=a}C_a^b\d_j\Re(Z^b \bar{\psi}_j Z^c\psi),\\
&\div   Z^a u=0,\\
&\begin{aligned}
i\d_t Z^a \psi+\De Z^a\psi=&-i\sum_{b+c=a} C_a^b Z^b (u+2A)\cdot \nab Z^c\psi+\sum_{b+c=a}C_a^b Z^b A_{d+1} Z^c\psi\\
&+\sum_{b+c+e=a} C_a^{b,c} Z^b (A+u)\cdot Z^c A Z^e\psi\\
&-i\sum_{b+c=a} C_a^b \nab Z^b u\cdot Z^c\psi+i\sum_{b+c+e=a}C_a^{b,c} \Im(Z^b \psi_l Z^c\bar{\psi})Z^e\psi_l,
\end{aligned}\\
\end{aligned}
\right.
\end{equation}
with connection coefficients $A$ and $A_{d+1}$ satisfy
\begin{align}  \label{Ell-A}
    \De Z^a A&=\sum_{b+c=a}C_a^b \d_l \Im (Z^b\psi_l Z^c\bar{\psi}),\\ \label{Ell-Ad1}
    \De Z^a A_{d+1}&=-\sum_{b+c=a} C_a^b \Re\d_m(\d_l Z^b \psi_l Z^c \bar{\psi}_m)\\\nonumber
    &\quad +\sum_{b+c+e=a}C_a^{b,c}\Im\d_m(Z^b (u+A)\cdot Z^c\psi Z^e \bar{\psi}_m),
\end{align}
where the coefficients $C_a^b$ and $C_a^{b,c}$ are
\[
C_a^b=\frac{a!}{b!(a-b)!},\quad C_a^{b,c}=\frac{a!}{b!c!(a-b-c)!}.
\]
We define the profile of $\psi$ by 
\begin{equation}\label{Psi-def}
\Psi=e^{-it\De} \psi.
\end{equation}

First, we prove quantitative estimates for the differentiated fields $\psi$ with respect to $\phi$.
\begin{lemma}[Bounds for $\psi$]\label{psi-by_phi-Lemma}
	With the notation in Proposition \ref{frame}, if the map $\phi$ has the additional property $\| \phi\|_{H_Q^{k+1}}+\sum_{|a|=1}\|\nab \Lambda^a\phi\|_{H^{k-2}}\leq \ep$ for any $k\geq 3$, then for its differentiated fields $\psi_m$ with $1\leq m\leq d$ we have the bounds 
	\begin{equation}\label{psi-by_phi}
		\|\psi\|_{H^k}+\sum_{|a|=1}\| \Lambda^a\psi\|_{H^{k-2}}\lesssim \|\phi\|_{H^{k+1}_Q}+\sum_{|a|=1}\|\nab \Lambda^a\phi\|_{H^{k-2}}.
	\end{equation}
	If $\| \phi\|_{H_Q^{k+1}}+\sum_{|a|=1}\|\nab Z^a\phi\|_{H^{k-2}}\leq \ep$ for $k\geq 3$, we also have
	\begin{equation}\label{psi-by_phiZ}
		\|\psi\|_{H^k}+\sum_{|a|=1}\| Z^a\psi\|_{H^{k-2}}\lesssim \|\phi\|_{H^{k+1}_Q}+\sum_{|a|=1}\|\nab Z^a\phi\|_{H^{k-2}}.
	\end{equation}
\end{lemma}
\begin{proof}
	The two bounds \eqref{psi-by_phi} and \eqref{psi-by_phiZ} are proved in a same way, here we only show the proof of \eqref{psi-by_phi}. We prove this bound by induction.

	First, we bound the term $\psi$. By $|v|=|w|=1$, \eqref{diff-field} and Sobolev embedding, we have
	\begin{align*}
		\|\psi\|_{L^2\cap L^\infty}\leq \|\nab\phi\|_{L^2\cap L^\infty}(\| v\|_{L^\infty}+\| w\|_{L^\infty})\leq \|\nab\phi\|_{H^k}.
	\end{align*}
Then \eqref{A,Ad+1} and \eqref{phi-v-w} combined this bound yield
	\begin{align*}
		\|A\|_{L^2}\lesssim \| P_{\leq 0}(\psi^2)\|_{L^1}+\| P_{> 0}(\psi^2)\|_{L^2}\lesssim \|\psi\|_{L^2}\|\psi\|_{L^2\cap L^\infty}\les  \|\nab\phi\|_{H^k},
	\end{align*}
	and 
	\begin{align} \label{bd-vw}
		\| \nab v\|_{L^2}+\| \nab w\|_{L^2}\lesssim\|(\phi,v,w)\|_{L^\infty} (\| \psi\|_{L^2}+\| A\|_{L^2})\lesssim \| \psi\|_{L^2}\les \|\nab\phi\|_{H^k}.
	\end{align}
 
	To prove \eqref{psi-by_phi} for $\psi$ in $H^k$, in view of the above bounds we can assume that 
	\begin{equation} \label{Ass-psin}
		\|\psi\|_{H^n}+\| A\|_{H^n}+\|\nab v\|_{H^n}+\|\nab w\|_{H^n}\lesssim \|\phi\|_{H^{k+1}_Q},\quad \text{for any }n<l\leq k.
	\end{equation}
	By this assumption, \eqref{diff-field} and $|\phi|=|v|=|w|=1$, we have
	\begin{align*}
		\|\nab^l \psi\|_{L^2}\lesssim &\ \|\nab^{l+1}\phi\|_{L^2}\|(v,w)\|_{L^\infty}+\sum_{l_1+l_2=l,0<l_2<l}\|\nab^{l_1+1}\phi\|_{L^4}\| \nab^{l_2}(v,w)\|_{L^4}\\
		&+\|\nab\phi\|_{L^\infty}\|\nab^l(v,w)\|_{L^2}\\
		\lesssim &\ \|\phi\|_{H^{k+1}_Q}(1+\|\nab(v,w)\|_{H^{l-1}})\lesssim  \|\phi\|_{H^{k+1}_Q}.
	\end{align*}
	Using the formula \eqref{A,Ad+1} to bound the $A$ by
	\begin{align*}
		\| \nab^l A\|_{L^2}\lesssim \| \nab^{l-1}(\psi^2)\|_{L^2}\lesssim \|\psi\|_{H^{l-1}}\|\psi\|_{L^\infty}\lesssim \|\phi\|_{H^{k+1}_Q}^2.
	\end{align*}
	Similarly, by \eqref{phi-v-w} we also have
	\begin{align*}
		\|\nab^{l+1}v\|_{L^2}\lesssim& \|\nab^l\psi\|_{L^2}+\sum_{l_1+l_2,l_2>0}\|\nab^{l_1}\psi\|_{L^4}\|\nab^{l_2}\phi\|_{L^4}+\| A\|_{H^l}\|\nab w\|_{H^{l-1}}\\
		\lesssim & \|\phi\|_{H^{k+1}_Q}+\| \psi\|_{H^l}^2\|\nab w\|_{H^{l-1}}\lesssim  \|\phi\|_{H^{k+1}_Q},
	\end{align*}
	and 
	\begin{align*}
		\|\nab^{l+1}w\|_{L^2}\lesssim  \|\phi\|_{H^{k+1}_Q}.
	\end{align*}
 Hence, the bound \eqref{Ass-psin} also holds for $n=l$. Then the bound \eqref{psi-by_phi} for $\psi$ in $H^k$ follows. 
	
	Next, we bound the term $\Lambda^a\psi$ for $|a|=1$. The operators $\tilde \Om$ and $\tilde S_0$ applied to tangent vector fields $v,\ w$ are denoted as 
	\begin{align*}
		\tilde{\Om}_i v=\Om_i v,\quad \tilde{\Om}_i w=\Om_i w,\quad \tilde{S}_0 v=x\cdot\nab v,\quad  \tilde{S}_0 w=x\cdot\nab w.
	\end{align*}
By \eqref{diff-field} and Sobolev embeddings we bound the $\Lambda^a \psi$ by
	\begin{align*}
		\sum_{|a|=1}\| \Lambda^a\psi\|_{L^2}\lesssim& \sum_{|a|=1}\| \La^a\nab  \phi\|_{L^2}+\sum_{|a|=1}\|\nab\phi\|_{L^3}\| (\La^av,\La^a w)\|_{L^6}\\
		\lesssim & \sum_{|a|\leq 1}\| \nab \La^a \phi\|_{L^2}+\sum_{|a|=1}\ep\|(\nab \La^a v,\nab \La^a w)\|_{L^2}.
	\end{align*}
From the formula \eqref{Ell-A}, we bound the $\Lambda^a A$ by
	\begin{align*}
		\sum_{|a|=1}\| \Lambda^a A\|_{L^2}\lesssim \sum_{|a|=1}\| \nab^{-1}(\Lambda^a\psi\cdot \psi)\|_{L^2}\lesssim \sum_{|a|=1}\| \Lambda^a\psi\|_{L^2}\|\psi\|_{L^2\cap L^\infty}.
	\end{align*}
Using the relation \eqref{phi-v-w}, we obtain 
	\begin{align*}
		\sum_{|a|=1}\|(\nab \Lambda^a v,\nab \Lambda^a w)\|_{L^2}&\lesssim \sum_{|a|=1}\| (\Lambda^a\psi,\Lambda^a A)\|_{L^2}+\|(\psi,A)\|_{L^3}\sum_{|a|=1}\|(\Lambda^a\phi,\Lambda^a v,\Lambda^a w)\|_{L^6}\\
		&\lesssim \sum_{|a|=1} \| \Lambda^a\psi\|_{L^2}+\ep \sum_{|a|=1}\|(\nab \Lambda^a\phi,\nab \Lambda^a v,\nab \Lambda^a w)\|_{L^2}.
	\end{align*}
	These imply
	\begin{align*}
		\sum_{|a|=1}\Big(\| \La^a \psi\|_{L^2}+\| \La^a A\|_{L^2}+\|(\nab \La^a v,\nab \La^a w)\|_{L^2}\Big)\lesssim \sum_{|a|=1}\| \nab \La^a \phi\|_{L^2}\lesssim \ep.
	\end{align*}
	With the bound at hand, we can assume that
	\begin{align}   \label{IndAss-n}
		\sum_{|a|=1}\Big(\| \La^a\psi\|_{H^n}+\| \La^a A\|_{H^n}+\|(\nab \La v,\nab \La w)\|_{H^n}\Big)\lesssim \ep,\quad \text{for any }n<l\leq k-2.
	\end{align}
By \eqref{diff-field}, \eqref{Ell-A}, \eqref{phi-v-w} and Sobolev embedding we have
	\begin{align*}
		\sum_{|a|=1}\|\nab^l \La^a\psi\|_{L^2}\lesssim& \sum_{l_1+l_2=l}\|\nab^{l_1}\La^a \nab\phi\|_{L^2}\|\nab^{l_2}(v,w)\|_{L^\infty}\\
		&+\sum_{l_1+l_2=l,l_2>0}\sum_{|a|=1}\|\nab^{l_1+1}\phi\|_{L^\infty}\|(\nab^{l_2}\La^a v,\nab^{l_2}\La^a w)\|_{L^2}\\
		&+\|\nab^{l+1}\phi\|_{L^3}\sum_{|a|=1}\|(\La^a v,\La^a w)\|_{L^6}\\
		\lesssim &\ \ep+\|\nab^l \phi\|_{H^{l+1}}\sum_{|a|=1}\|(\nab \La^a v,\nab \La^a w)\|_{L^2}\lesssim \ep,
	\end{align*}
	\begin{align*}
		\sum_{|a|=1}\| \nab^l \La^a A\|_{L^2}\lesssim \sum_{|a|=1}\|\nab^{l-1}(\La^a\psi\cdot\psi)\|_{L^2}\lesssim \sum_{|a|=1}\| \La^a \psi\|_{H^{l-1}}\|\psi\|_{H^{l+1}}\lesssim \ep^2 ,
	\end{align*}
	and 
	\begin{align*}
		\sum_{|a|=1}\|\nab^{l+1} (\La^a v,\La^a w)\|_{L^2}&\lesssim  \|\nab^{l+1}(v,w)\|_{L^2}+\sum_{|a|=1}\|\nab^l \La^a(\psi\phi_0+A(v,w))\|_{L^2}\\
		&\lesssim \ep+ \sum_{l_1+l_2=l}\sum_{|a|=1}\|\nab^{l_1}(\La^a \psi,\La^a  A)\|_{L^2}\|\nab^{l_2}(\phi,v,w)\|_{L^\infty}\\
		&\quad +\sum_{l_1+l_2=l,l_2>0}\sum_{|a|=1}\|\nab^{l_1}(\psi,A)\|\|\nab^{l_2}(\La^a\phi,\La^av,\La^aw)\|_{L^2}\\
		&\quad +\|\nab^{l}(\psi,A)\|_{L^3}\sum_{|a|=1}\|(\La^a\phi,\La^a v,\La^a w)\|_{L^6}\\
		&\lesssim  \ep+\ep\sum_{|a|=1}\| \La^a\psi\|_{H^{l}}\lesssim \ep.
	\end{align*}
    These bounds imply that the bound \eqref{IndAss-n} also hold for $n=l$. Hence, we obtain the estimate \eqref{psi-by_phi} and complete the proof of the lemma.
\end{proof}

The bound \eqref{MainAss_ini} combined with the above Lemma \ref{psi-by_phi-Lemma} deduces the initial data $(u_0,\psi_0)$ for the differentiated system \eqref{diff-field}, which satisfies
\begin{equation}    \label{ini-diffSys}
	\|(u_0,\psi_0)\|_{H^3}+\sum_{|a|=1} \lV (\La^a u_0,\La^a\psi_0)\rV_{H^1} \les \epsilon_0.
\end{equation}
Now we state our main bootstrap proposition as follows:
\begin{prop}[Bootstrap proposition]\label{Main_Prop}
    Assume that $(u,\psi)$ is a solution to \eqref{sys-2}-\eqref{A,Ad+1} on some time interval $[0,T]$, $T\geq 1$ with initial data satisfying the assumption \eqref{ini-diffSys}.
    Assume also that the solution $(u,\psi)$ and the profile $\Psi=e^{-it\De}\psi$ satisfy the bootstrap hypothesis
    \begin{gather}           \label{Main_Prop_Ass1}
    \begin{aligned}
        \sup_{t\in[0,T]}\Big\{&\|(u,\psi)\|_{H^3}+\|\nab u\|_{L^2([0,t]:H^3)}\\
        &+\sum_{|a|=1}\Big(\|(Z^a u,Z^a\psi)\|_{H^1}+\|\nab Z^a u\|_{L^2([0,t]:H^1)}\Big)\Big\}\leq\epsilon_1, 
    \end{aligned} \\ \label{Main_Prop_Ass2}
        \sup_{t\in[0,T]}\lV x\c\nab \Psi\rV_{H^1} \leq \ep_1,
    \end{gather} 
    where $\ep_1=\ep_0^{2/3}$. Then the following improved bounds hold
    \begin{gather}\label{Main_Prop_result1}
    \begin{aligned}
        \sup_{t\in[0,T]}\Big\{&\|(u,\psi)\|_{H^3}+\|\nab u\|_{L^2([0,t]:H^3)}\\
        &+\sum_{|a|=1}\Big(\|(Z^a u,Z^a\psi)\|_{H^1}+\|\nab Z^a u\|_{L^2([0,t]:H^1)}\Big)\Big\}\les\epsilon_0, 
    \end{aligned}
          \\\label{Main_Prop_result2}
        \sup_{t\in[0,T]}\lV x\c\nab \Psi\rV_{H^1} \lesssim \ep_0.
    \end{gather}
\end{prop}

From this proposition and continuity method, the global existence of \eqref{sys-2}-\eqref{A,Ad+1} is obtained immediately. Hence, in the next sections, we will concern on the proofs of Proposition \ref{Main_Prop} and Theorem \ref{Main_thm-ori}.

\bigskip 
\section{Decay of velocity field and differentiated field}
In this section, we give the various decay estimates of $u$ and $\psi$ under the bootstrap hypothesis \eqref{Main_Prop_Ass1} and \eqref{Main_Prop_Ass2}, which will play key roles in the energy estimates in
the next sections. Here we start with the basic decay estimates.

\begin{lemma}[\cite{HaMiNa99}, Lemma 2.4]
    For any Schwartz function $f\in \mathcal{S}(\R^3)$ and $t>1$, we have
    \begin{align}   \label{df-decay}
        &\|\nab f\|_{L^6}\lesssim t^{-1}\| \Theta f\|_{L^2},\\   \label{f-decay}
        &\|f\|_{L^\infty}\lesssim t^{-\frac{3}{4}}\| \Theta f\|_{L^2}^{\frac{3}{4}}\|f\|_{L^2}^{\frac{1}{4}},
    \end{align}
    where $\Theta =(x\c\nab+2it\De,\Om)$.
\end{lemma}

We then use these two estimates to give the decays of $\psi$.
\begin{lemma}[Decay of fields $\psi$]  \label{Dec_psi-Lem}
    With the notations and hypothesis in Proposition \ref{Main_Prop}, for any $t\in[0,T]$ we have
	\begin{gather}     \label{Dec_dpsi}
	\lV \nab\psi(t)\rV_{W^{1,6}} \lesssim \ep_1 \<t\>^{-1} ,\\ \label{Dec_psi}
	\lV\psi(t)\rV_{W^{\frac{5}{4},\infty}}\lesssim \ep_1\<t\>^{-3/4}.
	\end{gather}
\end{lemma}
\begin{proof}
    We prove the bound \eqref{Dec_dpsi} first. By \eqref{df-decay}, we have
    \[ \|\nab  \<\nab\>\psi\|_{L^6}\les t^{-1} \big(\|(x\c\nab+2it\De)\<\nab\>\psi\|_{L^2}+\|\Om \<\nab\>\psi\|_{L^2} \big)   \]
    In view of the relation $\Psi=e^{-it\De}\psi$ and the commutator 
    \begin{equation}          \label{Comm}
    (x\cdot\nab+2it\De)e^{it\De}\Psi=e^{it\De}(x\cdot\nab)\Psi,
    \end{equation}
    we have for $t>1$ 
    \begin{align*}
        \| \nab  \<\nab\>\psi\|_{L^6}\lesssim &\  t^{-1}(\| (x\cdot \nab+2it\De)e^{it\De}\<\nab\>\Psi\|_{L^2}+\| \Om \<\nab\>\psi\|_{L^2})\\
        \lesssim &\  t^{-1} (\| x\cdot \nab \<\nab\>\Psi\|_{L^2}+\ep_1),
    \end{align*}
    Then the bound \eqref{Dec_dpsi} follows by \eqref{Main_Prop_Ass2}. 
    
    Next, we prove the bound \eqref{Dec_psi}. For low-frequency part $P_{<0}\psi$, by Sobolev embedding and \eqref{Dec_dpsi} we easily have
    \[  \|P_{<0}\psi\|_{L^\infty}\les  \|P_{<0}\psi\|_{L^6}\les \ep_1 t^{-1}. \]
    For high-frequency part $P_k \psi$ with $k>0$, using the estimate \eqref{f-decay} and commutator \eqref{Comm} to yield
    \begin{align*}
        2^{5k^+/4}\|P_k \psi\|_{L^\infty}&\les t^{-3/4} 2^{3k^+/4}\big(\|(x\c\nab+2it\De)P_k\psi\|_{L^2}+\|\Om P_k\psi\|_{L^2} \big)^{3/4}\\&\quad \cdot(2^{3k^+}\|P_k\psi\|_{L^2})^{1/4}2^{-k^+/4}\\
        &\les t^{-3/4} 2^{3k^+/4}\big(\|(x\c\nab )P_k\Psi\|_{L^2}+\|\Om P_k\psi\|_{L^2} \big)^{3/4}\|\psi\|_{H^3}^{1/4}2^{-k^+/4}.
    \end{align*}  
Since 
\begin{align*}
    &2^{k^+}\big(\|(x\c\nab )P_k\Psi\|_{L^2}+\|\Om P_k\psi\|_{L^2}\big)\\
    &\les 2^{k^+}\big(\|P_k(x\c\nab )\Psi\|_{L^2}+\|P_k\tilde\Om \psi\|_{L^2}+\|P_k \psi\|_{L^2}\big)\\
    &\les \|(x\c\nab )\Psi\|_{H^1}+\|\tilde\Om \psi\|_{H^1}+\|\psi\|_{H^1}\les \ep_1,
\end{align*}
then we have
    \begin{align*}
        2^{5k^+/4}\|P_k \psi\|_{L^\infty}&\les t^{-3/4} \ep_1^{3/4}\|\psi\|_{H^3}^{1/4}2^{-k^+/4}
        \les t^{-3/4} \ep_1 2^{-k^+/4}  .
    \end{align*}
    Sum over $k$, the estimate \eqref{Dec_psi} follows.
    Hence this completes the proof of the lemma.
\end{proof}

Next, we use \eqref{Dec_heat} and Lemma \ref{Dec_psi-Lem} to prove the decays of velocity $u$. Applying the Leray projection $\P=I_d+\nab(-\De)^{-1}\nab$ to the $u$-equation in \eqref{sys-2}, then by Duhamel's formula, the solution $u$ can be expressed as 
\begin{align}  \label{Duh-u}
    u(t)=e^{t\De}u_0-\int_0^t e^{(t-s)\De} \P\big(u\cdot\nab u+\d_j \Re(\psi\bar\psi_j)\big)\  ds.
\end{align}

\begin{lemma}       \label{Dec_lem}
	With the notations and hypothesis in Proposition \ref{Main_Prop}, for any $t\in[0,T]$, we have
	\begin{gather}            \label{Dec_u}
	\lV u\rV_{W^{5/4,\infty}}\lesssim \ep_1\<t\>^{-3/4}.
	\end{gather}
\end{lemma}
\begin{proof}
    From the formula \eqref{Duh-u}, we need to bound the right hand side of \eqref{Duh-u} in $W^{5/4,\infty}$. The first term $e^{t\De}u_0$ is estimated by \eqref{Dec_heat} and \eqref{ini-diffSys}. For the integral in \eqref{Duh-u}, we claim
    \begin{align}   \label{Dec_u-key}
        \int_0^t \|e^{(t-s)\De}\P(u\cdot \nab u+\nab(\psi^2))\|_{W^{5/4,\infty}}\ ds\lesssim \ep_1^2\<t\>^{-3/4}+\ep_1\sup_{s\in[t/2,t]}\|u(s)\|_{L^\infty}.
    \end{align}
    By \eqref{Dec_heat}, $\div u=0$ and Sobolev embedding we bound the first term by
    \begin{align*}
        \int_0^t \|e^{(t-s)\De}\P(u\cdot \nab u)\|_{W^{5/4,\infty}}\ ds\lesssim &\  \int_0^{t/2} (t-s)^{-2} \| u^2\|_{W^{5/4,1}}ds\\
        &\ +\int_{t/2}^t (t-s)^{-1+\de}\<t-s\>^{-1/4-\de}\|u^2\|_{H^{7/4+2\de}}ds\\ 
        \lesssim &\ \ep_1^2\<t\>^{-1}+\ep_1\sup_{s\in[t/2,t]}\|u(s)\|_{L^\infty}.
    \end{align*}
    Similarly, by \eqref{Dec_psi} we easily have
    \begin{align*}
        \int_0^t \|e^{(t-s)\De}\P\nab(\psi^2)\|_{W^{5/4,\infty}}ds\lesssim  &\ \ep_1^2\<t\>^{-1}+\ep_1\sup_{s\in[t/2,t]}\|\psi(s)\|_{L^\infty}\lesssim \ep_1^2\<t\>^{-3/4}.
    \end{align*}
    These give the bound \eqref{Dec_u-key}, and conclude the decay estimate \eqref{Dec_u}.
\end{proof}

We then use this decay estimate to prove the following more important decay estimate for $\nab u$.
\begin{lemma}  \label{Dec_lem_VF}
	With the notations and hypothesis in Proposition \ref{Main_Prop}. For any $t\in [0,T]$, we have
	\begin{equation}        \label{Dec_du}
	\lV \nab u\rV_{W^{1,\infty}}\lesssim \ep_1\<t\>^{-5/4}.
	\end{equation}
\end{lemma}
\begin{proof}
    From Duhamel's formula \eqref{Duh-u}, the term $e^{t\De}\nab u_0$ is estimated by \eqref{Dec_heat} and \eqref{ini-diffSys}. Here it remains to prove
    \begin{align}   \label{Dec_du-key}
        \int_0^t \|e^{(t-s)\De}\P\nab(u\cdot \nab u+\nab(\psi^2))\|_{W^{1,\infty}}ds\lesssim \ep_1^2\<t\>^{-5/4}.
    \end{align}
    By \eqref{Dec_heat}, $\div u=0$ and Sobolev embedding, we have
    \begin{align*}
        \int_0^t \|e^{(t-s)\De}\P\nab(u\cdot \nab u)\|_{W^{1,\infty}}ds\lesssim &\ \int_0^{t/2} (t-s)^{-5/2} \| u^2\|_{W^{1,1}}ds\\
        &+\int_{t/2}^t (t-s)^{-1+\de}\|u^2\|_{W^{1+2\de,\infty}}ds\\ 
        \lesssim &\ \ep_1^2\<t\>^{-3/2}+\<t\>^\de \sup_{s\in[t/2,t]}\|u(s)\|_{W^{1+2\de,\infty}}^2\\
        \lesssim &\ \ep_1^2\<t\>^{-3/2+\de}.
    \end{align*}
    Similar, by \eqref{Dec_psi} we easily have
    \begin{align*}
        \int_0^t \|e^{(t-s)\De}\P\nab^2(\psi^2)\|_{W^{1,\infty}}ds\lesssim  &\ \ep_1^2\<t\>^{-3/2}+\<t\>^\de\sup_{s\in[t/2,t]}\|\psi(s)\|_{W^{1+2\de,\infty}}^2\lesssim \ep_1^2\<t\>^{-3/2+\de}.
    \end{align*}
    These give the bound \eqref{Dec_du-key}, and conclude the decay estimate \eqref{Dec_du}.
\end{proof}

With the above decays of $\psi$ and $u$, we then prove the decays for $A$ \eqref{Ell-A} and $A_{d+1}$ \eqref{Ell-Ad1}.

\begin{cor}[Decays of $A$ and $A_{d+1}$]  \label{Dec_A_Cor}
    With the notations and hypothesis in Proposition \ref{Main_Prop}, for the elliptic equations \eqref{Ell-A}, \eqref{Ell-Ad1} and any $t\in[0,T]$, we have
	\begin{align}  \label{Dec_d1/2-A}
    \| |\nab|^{\frac{1}{2}} A\|_{H^3}+\sum_{|a|=1}\| |\nab|^{\frac{1}{2}}Z^aA\|_{H^1}&\lesssim \ep_1^2\<t\>^{-1/2},\\\label{Dec_dA}
    \| \nab A\|_{W^{1,\infty}}&\lesssim \ep_1^2\<t\>^{-3/2+3\de/4},\\   \label{Dec_Ad+1-L2}
    \| A_{d+1}\|_{H^3}+\sum_{|a|=1}\| Z^aA_{d+1}\|_{H^1}&\lesssim \ep_1^2\<t\>^{-1/2},\\\label{Dec_Ad+1-infty}
    \| A_{d+1}\|_{W^{1,\infty}}&\lesssim \ep_1^2\<t\>^{-5/4+3\de},
	\end{align}
 where $0<\de\ll 1$ is sufficiently small.
\end{cor}
\begin{proof}
    We start with the first estimate \eqref{Dec_d1/2-A}. By \eqref{Ell-A} and \eqref{Dec_psi} we have 
    \begin{align*}
        \| |\nab|^{1/2} A\|_{H^3}\lesssim &\ \||\nab|^{-1/2}(\psi^2)_{\leq 0}\|_{L^2}+\|\psi^2\|_{H^3}\\
        \lesssim &\ \|(\psi^2)_{\leq 0}\|_{L^{3/2}}+\|\psi\|_{H^3}\|\psi\|_{L^{\infty}}\\
        \lesssim &\ \|\psi\|_{L^2}^{4/3}\|\psi\|_{L^\infty}^{2/3}+\ep_1^2\<t\>^{-3/4}\\
        \lesssim &\ \ep_1^2\<t\>^{-1/2}.
    \end{align*}
    And similarly, for $|a|=1$ we have
    \begin{align*}
        \| |\nab|^{1/2}Z^aA\|_{H^1}\lesssim &\ \||\nab|^{-1/2}(Z^a\psi\psi)_{\leq 0}\|_{L^2}+\|Z^a\psi\psi\|_{H^1}\\
        \lesssim &\ \|(Z^a\psi\psi)_{\leq 0}\|_{L^{3/2}}+\|Z^a\psi\|_{H^1}\|\psi\|_{W^{1,\infty}}\\
        \lesssim &\ \|Z^a\psi\|_{L^2}\|\psi\|_{L^2}^{1/3}\|\psi\|_{L^\infty}^{2/3}+\ep_1^2\<t\>^{-3/4}\\
        \lesssim &\ \ep_1^2\<t\>^{-1/2}.
    \end{align*}
    Thus the bound \eqref{Dec_d1/2-A} follows.
    
    We use Sobolev embedding and \eqref{Dec_psi} to give the second estimate \eqref{Dec_dA} 
    \begin{align*}
        \| \nab A\|_{W^{1,\infty}}\lesssim &\ \| \mathcal R(\psi^2)\|_{W^{1+3\de,2/\de}}\\
        \lesssim &\ \|\psi\|_{W^{1+3\de,\infty}}\|\psi\|_{W^{1+3\de,2/\de}}\\ 
        \lesssim &\ \ep_1^2\<t\>^{-3/2+3\de/4},
    \end{align*}
where $\mathcal R=\frac{\d_x}{|\nab|}$ is the Riesz transform.

    For the third bound \eqref{Dec_Ad+1-L2}, when $|a|=0$, by \eqref{Ell-Ad1}, Sobolev embedding, \eqref{Dec_dpsi} and \eqref{Dec_psi} we have
    \begin{align*}
        \| A_{d+1}\|_{H^{3}}\lesssim &\  \| \nab^{-1}(\nab\psi \psi+(u+A)\psi^2)\|_{H^{3}}\\
        \lesssim &\ \|(\nab\psi \psi+(u+A)\psi^2)_{\leq 0}\|_{L^{6/5}}+\|(\nab\psi \psi+(u+A)\psi^2)_{> 0}\|_{H^{3-1}}\\
        \lesssim &\ \| \nab\psi\|_{L^3}\|\psi\|_{L^2}+\|(u,A)\|_{L^2}\|\psi\|_{L^3}\|\psi\|_{L^\infty}+\| \psi\|_{W^{1,\infty}}\|\psi\|_{H^{3}}\\
        & +\|(u,A)\|_{H^{3}}\|\psi\|_{L^\infty}^2+ \|\psi\|_{H^{3}}\|(u,A)\|_{L^\infty}\|\psi\|_{L^\infty}\\
        \lesssim &\ \ep_1^2\<t\>^{-1/2}+\ep_1^3\<t\>^{-1}+\ep_1^2\<t\>^{-3/4}+\ep_1^3\<t\>^{-3/2}\lesssim \ep_1^2\<t\>^{-1/2}.
    \end{align*}
    When $|a|=1$, from the equation \eqref{Ell-Ad1} we have
    \begin{align*}
        \sum_{|a|=1}\| Z^aA_{d+1}\|_{H^{1}}\lesssim &\  \| \nab^{-1}(\sum_{|b+c|=1}\nab Z^b\psi Z^c\psi+\sum_{|b+c+e|=1}(Z^bu+Z^bA)Z^c\psi Z^e\psi)\|_{H^{1}}.
    \end{align*}
    By \eqref{Dec_dpsi} and \eqref{Dec_psi}, the first term in the right-hand side can be bounded by
    \begin{align*}
        &\sum_{|a|=1}\| \nab^{-1}(\nab Z^a\psi \psi+\nab\psi Z^a\psi)\|_{H^{1}}\\
        \lesssim &\ \sum_{|a|=1}\| \mathcal R(Z^a\psi\psi)+ \nab^{-1}(\nab\psi Z^a\psi)\|_{H^{1}}\\
        \lesssim &\ \sum_{|a|=1}\Big(\|Z^a\psi\|_{H^{1}}\|\psi\|_{W^{1,\infty}}+ \|(\nab\psi Z^a\psi)_{\leq 0}\|_{L^{6/5}}+\|(\nab\psi Z^a\psi)\|_{H^{1-1}}\Big)\\
        \lesssim &\ \ep_1^2\<t\>^{-3/4}+\sum_{|a|=1}\big(\|\nab\psi\|_{L^3}\|Z^a\psi\|_{L^2}+\|\nab\psi\|_{W^{1,6}}\|Z^a\psi\|_{H^{1}}\big)\\
        \lesssim &\ \ep_1^2\<t\>^{-3/4}+\ep_1^2\<t\>^{-1/2}+\ep_1^2\<t\>^{-1}\lesssim \ep_1^2\<t\>^{-1/2}.
    \end{align*}
    We use Littlewood-Paley decomposition to write the second term as 
    \begin{align}\label{SecTerm}
        &\| \nab^{-1}\sum_{|b+c+e|=1}(Z^bu+Z^bA)Z^c\psi Z^e\psi\|_{H^{1}}\\\nonumber
        \lesssim &\ \|(\sum_{|b+c+e|=1}(Z^bu+Z^bA)Z^c\psi Z^e\psi)_{\leq 0}\|_{L^{6/5}}\\\nonumber
        &\ +\|(\sum_{|b+c+e|=1}(Z^bu+Z^bA)Z^c\psi Z^e\psi)_{> 0}\|_{L^2},
    \end{align}
    by \eqref{Dec_psi}, \eqref{Dec_u} and \eqref{A_m-Eq}, which can be bounded by 
    \begin{align*}
        \eqref{SecTerm}\lesssim &\ \sum_{|a|=1}\big(\|Z^a(u,A)\|_{L^2}\|\psi\|_{L^3}\|\psi\|_{L^\infty}+\|(u,A)\|_{L^3}\|Z^a\psi\|_{L^2}\|\psi\|_{L^\infty}\\
        &\ +\|Z^a(u,A)\|_{H^{1}}\|\psi\|_{W^{1,\infty}}^2+\|(u,A)\|_{W^{1,\infty}}\|\psi\|_{W^{1,\infty}}\|Z^a\psi\|_{H^{1}}\big)\\
        \lesssim &\ \ep_1^3\<t\>^{-1}+\ep_1^3\<t\>^{-3/2}\lesssim \ep_1^3\<t\>^{-1}.
    \end{align*}
    
    Finally, we prove the last bound \eqref{Dec_Ad+1-infty}. By \eqref{Dec_dpsi} and \eqref{Dec_psi} we have 
    \begin{align*}
        \| A_{d+1}\|_{W^{1,\infty}}\lesssim &\  \| \nab^{-1}\mathcal R(\nab\psi \psi+(u+A)\psi^2)\|_{W^{1,\infty}}\\
        \lesssim &\ \| (\nab\psi \psi+(u+A)\psi^2)_{\leq 0}\|_{L^{\frac{3}{1+3\de}}}\\
        &\ +\| (\nab\psi \psi+(u+A)\psi^2)_{>0}\|_{W^{1,\frac{3}{1-3\de}}}\\
        \lesssim &\ (\|\nab\psi\|_{L^{\frac{3}{1+3\de}}} +\|(u,A)\|_{L^{\frac{3}{1+3\de}}}\|\psi\|_{L^\infty})\|\psi\|_{L^\infty}\\
        &\ +(\|\nab\psi\|_{W^{1,\frac{3}{1-3\de}}} +\|(u,A)\|_{W^{1,\frac{3}{1-3\de}}}\|\psi\|_{W^{1,\infty}})\|\psi\|_{W^{1,\infty}}\\
        \lesssim &\ \ep_1^2\<t\>^{-5/4+3\de}+\ep_1^3\<t\>^{-3/2}
        \lesssim \  \ep_1^2\<t\>^{-5/4+3\de}.
    \end{align*}
    This completes the proof of the lemma.
\end{proof}

From the above decays for $\psi$ and $A$, the Coulomb frame $v,\ w$ near the fixed vectors $v_\infty$ and $w_\infty$ also admit some decay properties in $W^{1,\infty}$. The following corollary will help us in the proof of \eqref{thm-scatter}.
\begin{cor}
With the notations and hypothesis in Proposition \ref{Main_Prop}, for the Coulomb frame $v,w$, we denote $v_\infty=\lim_{|x|\rightarrow\infty}v(x)$ and $w_\infty=\lim_{|x|\rightarrow\infty}w(x)$. Then we have
    \begin{align}  \label{v-cong}
        \|v-v_\infty\|_{W^{1,\infty}}+\|w-w_\infty\|_{W^{1,\infty}}\les \ep_1\<t\>^{-\frac{1}{6}+\de}.
    \end{align}
\end{cor}
\begin{proof}
    By \eqref{phi-v-w} and $|\phi|=|v|=|w|=1$, we easily have
    \begin{align} \label{dv-inf}
        \|\nab(v-v_\infty)\|_{W^{1,\infty}}+\|\nab(w-w_\infty)\|_{L^{\infty}}&=\|\nab v\|_{W^{1,\infty}}+\|\nab w\|_{L^{\infty}}\\\nonumber
        &\les \|(\phi,w,v)\|_{L^\infty}\|(\psi,A)\|_{L^\infty}\\\nonumber
        &\les \|(\psi,A)\|_{L^\infty}.
    \end{align}
    Then from \eqref{Dec_psi} and \eqref{Dec_d1/2-A}, we bound by
    \begin{align}  \label{dv-inf-fin}
        \eqref{dv-inf}&\les \|\psi\|_{L^\infty}+\||\nab|^{1/2}A\|_{H^3}\les \ep_1\<t\>^{-3/4}+\ep_1\<t\>^{-1/2}\les \ep_1\<t\>^{-1/2}.
    \end{align}

    Next, we bound the $v-v_\infty$ and $w-w_\infty$ in $L^\infty$. Using Bernstein's inequality and interpolation inequality, for $\de_1>0$ small we have
    \begin{align*}
    \|v-v_\infty\|_{L^\infty}&\les \sum_{j\leq 0} 2^{j\frac{d}{d-\de_1}}\|P_j(v-v_\infty)\|_{L^{d-\de_1}}+\sum_{j> 0} 2^{j\frac{d}{d+\de_1}}\|P_j(v-v_\infty)\|_{L^{d+\de_1}}\\
    &\les \sum_{j\leq 0} 2^{j\frac{\de_1}{d-\de_1}}\|P_j\nab(v-v_\infty)\|_{L^{d-\de_1}}+\sum_{j> 0} 2^{-j\frac{\de_1}{d+\de_1}}\|P_j\nab(v-v_\infty)\|_{L^{d+\de_1}}\\
    &\les \|\nab(v-v_\infty)\|_{L^2}^{\frac{2}{d-\de_1}}\|\nab(v-v_\infty)\|_{L^\infty}^{1-\frac{2}{d-\de_1}}\\
    &\quad +\|\nab(v-v_\infty)\|_{L^2}^{\frac{2}{d+\de_1}}\|\nab(v-v_\infty)\|_{L^\infty}^{1-\frac{2}{d+\de_1}}.
    \end{align*}
    By \eqref{bd-vw} and \eqref{dv-inf-fin}, we can bound the above by
    \begin{align} \nonumber
        \|v-v_\infty\|_{L^\infty}&\les \ep_1^{\frac{2}{d-\de_1}} (\ep_1\<t\>^{-1/2})^{1-\frac{2}{d-\de_1}}+\ep_1^{\frac{2}{d+\de}} (\ep_1\<t\>^{-1/2})^{1-\frac{2}{d+\de_1}}\\ \label{dv}
        &\les \ep_1\<t\>^{-\frac{1}{6}+\frac{\de_1}{9-3\de_1}}\les \ep_1\<t\>^{-\frac{1}{6}+\de},
    \end{align}
    where we choose $\frac{\de_1}{9-3\de_1}=\de$. In a same line, we also have
    \begin{align} \label{dw}
        \|w-w_\infty\|_{L^\infty}\les \ep_1\<t\>^{-\frac{1}{6}+\de}.
    \end{align}
    Hence, collecting the bounds \eqref{dv-inf-fin}, \eqref{dv} and \eqref{dw}, the estimate \eqref{v-cong} follows.
\end{proof}

\bigskip

We then prove the decay estimates for the nonlinearities
\begin{align*}
    \mathcal N^{(a)}:&=\sum_{|b+c|\leq |a|}(Z^bA_{d+1} Z^c\psi+\d_m Z^b u\cdot Z^c\psi)\\
    &\quad +\sum_{|b+c+e|\leq |a|} [(Z^bA+Z^bu)\cdot  Z^cA Z^e\psi
    +Z^b\psi Z^c\psi Z^e\psi].
\end{align*}

\begin{cor}[Decays of $\mathcal N^{(a)}$]  \label{Dec_NonL-Cor}
    With the notations and hypothesis in Proposition \ref{Main_Prop}, for any $t\in[0,T]$ we have
\begin{gather}     \label{Dec_NonL}
    \| \mathcal N^{(0)}(t)\|_{H^3}\lesssim \ep_1^2\<t\>^{-5/4+3\de}+\ep_1\<t\>^{-3/4}\|\nab u\|_{H^3},\\ \label{Dec_NonL1}
    \| \mathcal N^{(1)}(t)\|_{H^1}\lesssim \ep_1^2\<t\>^{-5/4+3\de}+\ep_1\<t\>^{-3/4}\sum_{|a|\leq 1}\|\nab Z^au\|_{H^1}.
\end{gather}
\end{cor}
\begin{proof}
    By \eqref{Dec_u}, \eqref{Dec_du}, \eqref{Dec_psi} and Lemma \ref{Dec_A_Cor} we bound $\NN^{(0)}$ by
    \begin{align*}
        \| \NN^{(0)}\|_{H^{3}}\lesssim &\  (\| (A_{d+1},\nab u)\|_{H^{3}}+\|(u,A)\|_{H^{3}} \|(u,A)\|_{L^\infty}) \|\psi\|_{L^\infty}\\
        &\ +(\| (A_{d+1},\nab u)\|_{L^\infty}+\|(u,A)\|_{L^\infty}^2 ) \|\psi\|_{H^{3}}+\|\psi\|_{H^{3}}\|\psi\|_{L^\infty}^2\\
        \lesssim &\  (\ep_1^2\<t\>^{-1/2+3\de}+\| \nab u\|_{H^{3}})\ep_1\<t\>^{-3/4}+\ep_1^2\<t\>^{-5/4+3\de}\\
        \lesssim & \ \ep_1^2\<t\>^{-5/4+3\de}+\ep_1\<t\>^{-3/4}\| \nab u\|_{H^{3}}.
    \end{align*}
    We could also bound $\NN^{(a)}$ for $|a|=1$ by
    \begin{align*}
        &\| \NN^{(a)}\|_{H^{1}}\\
        \lesssim &\  \sum_{|a|=1}(\| (Z^aA_{d+1},\nab Z^au)\|_{H^{1}}+\|(Z^au+ Z^aA)\|_{H^{1}} \|(u+A)\|_{W^{1,\infty}}) \|\psi\|_{W^{1,\infty}}\\
        &\ +\sum_{|a|=1}(\| (A_{d+1},\nab u)\|_{W^{1,\infty}}+\|(u+A)\|_{W^{1,\infty}}^2 ) \|Z^a\psi\|_{H^{1}}\\
        &\ +\sum_{|a|=1}\|Z^a\psi\|_{H^{1}}\|\psi\|_{W^{1,\infty}}^2\\
        \lesssim &\  (\ep_1^2\<t\>^{-1/2+3\de}+\sum_{|a|=1}\| \nab Z^au\|_{H^{1}})\ep_1\<t\>^{-3/4}+\ep_1^2\<t\>^{-5/4+3\de}\\
        \lesssim &\  \ep_1^2\<t\>^{-5/4+3\de}+\ep_1\<t\>^{-3/4}\sum_{|a|=1}\| \nab Z^au\|_{H^{1}}.
    \end{align*}
    This concludes the proof of the Corollary.
\end{proof}

Here we also need to bound the $\nab^2 u$ and $\FF^{-1} (|\xi|\nab_\xi \widehat{u}(\xi))$ in $H^{1}$.
\begin{lemma}  \label{Dec_lem_d2u}
	With the notations and hypothesis in Proposition \ref{Main_Prop}. For any $t\in [0,T]$, we have
	\begin{align}        \label{Dec_d2u}
	&\lV \nab^2 u\rV_{H^{1}}\lesssim \ep_0\<t\>^{-1},\\\label{xidxi-u}
	&\|\FF^{-1} (|\xi|\nab_\xi \widehat{u}(\xi))\|_{H^{1}}\lesssim \ep_1.
	\end{align}
\end{lemma}
\begin{proof}
    \emph{(i) We prove the first bound \eqref{Dec_d2u}}. By Duhamel's formula \eqref{Duh-u}, \eqref{Dec_heat} and \eqref{ini-diffSys}, it suffices to prove
    \begin{align}  \label{Dec_d2u-key}
        \int_0^t \| e^{(t-s)\De}\nab^2\P (u\c \nab u+\nab(\psi^2))\|_{H^{1}}ds\lesssim \ep_1^2\<t\>^{-5/4+3\de}+\ep_1\sup_{s\in[t/2,t]}\|\nab^2 u\|_{H^{1}}.
    \end{align}
    By \eqref{Dec_heat} and \eqref{Main_Prop_Ass1} we bound the integral $\int_0^{t/2}$ by
    \begin{align*}
        &\int_0^{t/2} \| e^{(t-s)\De}\nab^2\P (u\c \nab u+\nab(\psi^2))\|_{H^{1}}ds\\
        \lesssim &\ \int_0^{t/2} (t-s)^{-9/4}(\| u\|_{H^{1}}^2+\| \psi\|_{H^{1}}^2)ds
        \lesssim\ \ep_1^2 \<t\>^{-5/4}.
    \end{align*}
    We use \eqref{Dec_heat}, \eqref{Dec_u}, \eqref{Dec_du} and H\"older inequality to bound the integral with respect to $u\cdot\nab u$ by
    \begin{align*}
        &\int_{t/2}^t \| e^{(t-s)\De}\nab^2\P (u\c \nab u)\|_{H^{1}}ds\\
        \lesssim &\ \int_{t/2}^t (t-s)^{-1/2}(\|\nab u\|_{H^{1}}\|\nab u\|_{L^\infty}+\| u\|_{W^{1,\infty}}\|\nab^2 u\|_{H^{1}})ds\\
        \lesssim &\ \int_{t/2}^t (t-s)^{-1/2}(\ep_1^{3/2}\<s\>^{-5/4}\|\nab^2 u(s)\|_{H^{1}}^{1/2}+\ep_1\<s\>^{-3/4}\|\nab^2 u(s)\|_{H^{1}})ds\\
        \lesssim &\ \ep_1^2\<t\>^{-3/2}+\ep_1\sup_{s\in[t/2,t]}\|\nab^2 u\|_{H^{1}}
    \end{align*}
    For the integral with respect to $\nab(\psi^2)$, we also have
    \begin{align}   \label{nabpsi2}
        &\int_{t/2}^t \| e^{(t-s)\De}\nab^3\P (\psi^2)\|_{H^{1}}ds
        \lesssim \ \int_{t/2}^t (t-s)^{-1+\de/2}\| \nab^{\de}(\nab\psi \psi)(s)\|_{H^{1}}ds,
    \end{align}
    Using Littlewood-Paley decomposition, the above integrand $\| \nab^{\de}(\nab\psi \psi)(s)\|_{H^{1}}$ is decomposed into high-frequency part and low-frequency part. For the high-frequency part, by Littlewood-Paley dichotomy and \eqref{Dec_dpsi} we have
    \begin{align}  \nonumber
    	2^{(1+\de)k^+}\|P_k(\nab\psi \psi)\|_{L^2}\lesssim &\  2^{(1+\de)k^+}\Big(\|\nab\psi_k\|_{L^{\frac{6}{1+6\de}}}\| \psi_{<k}\|_{L^{\frac{3}{1-3\de}}}\\\nonumber
    	&\ +\|\nab\psi_{<k}\|_{L^6}\| \psi_{k}\|_{L^3}+\sum_{k_1>k}\|\nab\psi_{k_1}\|_{L^6}\| \psi_{k_1}\|_{L^3}\Big)\\  \label{psi2}
    	\lesssim &\ 2^{(1+\de)k^+}\Big(\|\nab\psi_k\|_{L^6}^{1-3\de}\|\nab\psi_k\|_{L^2}^{3\de}\| \psi\|_{L^\infty}^{\frac{1+6\de}{3}}\|\psi\|_{L^2}^{\frac{2-6\de}{3}}\\\nonumber
    	&\ +\ep_1\<s\>^{-1}\| \psi_{k}\|_{L^\infty}^{1/3}\| \psi_{k}\|_{L^2}^{2/3}+\sum_{k_1>k}\ep_1\<s\>^{-1}\| \psi_{k_1}\|_{L^\infty}^{1/3}\| \psi_{k_1}\|_{L^2}^{2/3}\Big),
    \end{align}
    using \eqref{Dec_dpsi} and \eqref{Dec_psi} again, which is further bounded by
    \begin{align*}
        \eqref{psi2}\lesssim &\ 2^{(1+\de)k^+}(2^{-(1-3\de)k-6\de k}\ep_1^2 \<s\>^{-5/4+3\de/2}\\
        &\ +2^{-7k/3}\ep_1^2\<s\>^{-5/4}+\sum_{k_1>k}2^{-7k_1/3}\ep_1^2\<s\>^{-5/4} )\\
        \lesssim &\ 2^{-2\de k^+}\ep_1^2\<s\>^{-5/4+3\de/2}+ 2^{(-\frac{4}{3}+\de)k^+}\ep_1^2\<s\>^{-5/4}\\
        &\ + \sum_{k_1>k} 2^{(1+\de)(k^+-k_1)}2^{(-\frac{4}{3}+\de)k_1^+}\ep_1^2\<s\>^{-5/4}.
    \end{align*}
    This gives 
    \begin{align*}
        \|P_{\geq 0}(\nab\psi \psi)\|_{H^{1+\de}}\lesssim\  \ep_1^2\<s\>^{-5/4+3\de/2}.
    \end{align*}
    For the low-frequency part, we easily have
    \begin{align*}
        \|P_{< 0}(\nab\psi \psi)\|_{H^{1+\de}}\lesssim\ \| \nab\psi\|_{L^6}\|\psi\|_{L^3}\lesssim \ep_1^2\<s\>^{-5/4}.
    \end{align*}
    From these two estimates and \eqref{nabpsi2}, we obtain
    \begin{align*}
        \int_{t/2}^t \| e^{(t-s)\De}\nab^3\P (\psi^2)\|_{H^{1}}ds
        \lesssim \ \ep_1^2\<t\>^{-5/4+2\de}.
    \end{align*}
    This concludes the bound \eqref{Dec_d2u-key}, and hence gives the decay estimate \eqref{Dec_d2u}.

    \medskip 
    \emph{(ii) We prove the second bound \eqref{xidxi-u}}. By $u$-equation and \eqref{Dec_d2u} we have
    \begin{align*}
        \|\d_t u\|_{H^{1}}\lesssim&  \ \|\nab^2 u\|_{H^{1}}+\| u\nab u+\psi\nab\psi\|_{H^{1}}\\
        \lesssim & \ \ep_1\<t\>^{-1}+\| u\|_{W^{1,\infty}}\|\nab u\|_{H^{1}}+\| \psi\|_{W^{1,3}}\| \nab\psi\|_{W^{1,6}}\\
        \lesssim &\ \ep_1\<t\>^{-1}+\ep_1^2\<t\>^{-5/4}\lesssim \ \ep_1\<t\>^{-1}.
    \end{align*}
    This combined with \eqref{xidxif} and \eqref{Main_Prop_Ass1} to yields
    \begin{align*}
        \|\FF^{-1} (|\xi|\nab_\xi \widehat{u}(\xi))\|_{H^{1}}\lesssim \ \sum_{|a|= 1}\|  Z^a u\|_{H^{1}}+t\|\d_t u\|_{H^{1}}+\|u\|_{H^{1}}
        \lesssim  \ \ep_1.
    \end{align*}
    We complete the proof of the lemma.
\end{proof}

Finally, We state the following useful lemma.
\begin{lemma}
    With the hypothesis in Proposition \ref{Main_Prop}, for any $t\in[0,T]$, we have the estimate for profile $\Psi=e^{-it\De}\psi$ 
    \begin{equation} \label{xidxi-psi}
        (\sum_k 2^{2k^+}\| |\xi|\nab_\xi \widehat{\Psi_k}\|_{L^2}^2)^{1/2}\lesssim \ep_1.
    \end{equation}
\end{lemma}
\begin{proof}
    By \eqref{xidxif}, \eqref{Main_Prop_Ass1} and \eqref{Main_Prop_Ass2} we bound this by
    \begin{align*}
        (\sum_k 2^{2k^+}\| |\xi|\nab_\xi \widehat{\Psi_k}\|_{L^2}^2)^{1/2}\lesssim &\  \|\FF^{-1}(|\xi|\nab_\xi \widehat{\Psi})\|_{H^{1}}+\|\psi\|_{H^{1}}\\
        \lesssim &\  \|x\cdot\nab_x\Psi\|_{H^{1}}+\| \Om\Psi\|_{H^{1}}+\|\psi\|_{H^{1}}\lesssim \ep_1.
    \end{align*}
\end{proof}

\bigskip 
\section{Proof of Proposition \ref{Main_Prop} and Theorem \ref{Main_thm-ori}}
In this section we prove the bootstrap proposition \ref{Main_Prop}, and then prove the main theorem \ref{Main_thm-ori}. Here we start with the energy estimates \eqref{Main_Prop_result1}.

\begin{prop}       
\label{Prop_Ene_Sob}
With the notation and hypothesis in Proposition $\ref{Main_Prop}$, for any $t\in[0,T]$, we have
\begin{equation*}
\|(u,\psi)\|_{H^3}+\|\nab u\|_{L^2([0,t]:H^3)}
+\sum_{|a|=1}\Big(\|(Z^a u,Z^a\psi)\|_{H^1}+\|\nab Z^a u\|_{L^2([0,t]:H^1)}\Big)\les\epsilon_0.
\end{equation*}
\end{prop}
\begin{proof}
    We start with the energy estimates of velocity $Z^au$ for $|a|\leq 1$. Define the energy functional by
    \begin{equation*}
        E^a_u(t):=\frac{1}{2}\| Z^au\|_{H^{3-2|a|}}^2+\int_0^t \| \nab Z^au\|_{H^{3-2|a|}}^2ds.
    \end{equation*}
    Then by $u$-equation in \eqref{sys-VF}, $\div Z^b u=0$, integration by parts and Sobolev embeddings, we have
    \begin{align*}
        \frac{d}{dt}E^a_u(t)\lesssim &\  \sum_{|n|\in \{0,\ 3-2|a|\}}\sum_{b+c=a}\Big|\int \d^n Z^a u\cdot\d^n\big(Z^bu\cdot \nab Z^cu+\d_j(Z^b\psi Z^c\psi)\big) dx\Big|\\
        \les  &\  \sum_{|n|\in \{0,\ 3-2|a|\}}\sum_{b+c=a}\Big|\int \d_j\d^n Z^a u\cdot\d^n\big(Z^bu_j\cdot Z^cu+(Z^b\psi Z^c\psi)\big) dx\Big|\\
        \lesssim &\  \| \nab Z^a u\|_{H^{3-2|a|}}\Big(\sum_{|b+c|\leq 1}\| \nab  Z^b u\|_{H^{3-2|b|}}\|   Z^cu\|_{H^{3-2|c|}}\\
        &\quad +\sum_{|b|\leq 1}\|  Z^b\psi\|_{H^{3-2|a|}}\| \psi\|_{W^{1,\infty}}\Big).
    \end{align*}
    By the assumption \eqref{Main_Prop_Ass1}, \eqref{Dec_dpsi} and H\"older inequality we obtain
    \begin{align*}
        E^a_u(t)&\lesssim E^a_u(0)+\int_0^t \d_s E^a_u(s)ds\\
        &\lesssim \ep_0^2+\ep_1^3+\int_0^t \ep_1^2 \<s\>^{-3/4}\| \nab Z^a u(s)\|_{H^{3-2|a|}}ds\lesssim \ep_0^2.
    \end{align*}
    
    \bigskip
    
    Next, we prove the energy estimate of differentiated field $\psi$. Let's recall the nonlinearities $\mathcal N$ in $\psi$-equation
    \[
    \mathcal N:=(A_{d+1}+(A_l^2+u\c A))\psi_m-i\d_m u\c \psi+i\Im(\psi_l\bar{\psi}_m)\psi_l.
    \]
    By $\psi$-equation and $\div u=\div A=0$ we have
    \begin{align*}
        \frac{d}{dt}\|\psi\|^2_{H^{3}}&= \sum_{|n|\in\{0,3\}}\Big[-\Re\<\d^n \psi,\d^n((u+2A)\cdot \nab \psi) \>+\Re i\< \d^n\psi,\d^n \mathcal N\>\Big]\\
        &\lesssim  \| \psi\|_{H^{3}}(\|\nab(u+A)\cdot \psi\|_{H^{3}}+\|\mathcal N\|_{H^{3}}).
    \end{align*}
    We use \eqref{Dec_du} and \eqref{Dec_dpsi} to bound the second term by
    \begin{align*}
        \|\nab(u+A)\cdot \psi\|_{H^{3}}&\lesssim  \|\nab(u+A)\|_{L^\infty}\|\psi\|_{H^{3}}+\|\nab(u+A)\|_{H^{3}}\|\psi\|_{L^\infty}\\
        &\lesssim  \ep_1^2\<t\>^{-5/4}.
    \end{align*}
    By the above two estimate and \eqref{Dec_NonL} and \eqref{Main_Prop_Ass1}, we obtain 
    \begin{align*}
        \frac{d}{dt}\|\psi\|_{H^3}^2\les \ep_1 ( \ep_1^2\<t\>^{-5/4+3\de}+\ep_1\<t\>^{-3/4}\|\nab u\|_{H^{3}}).
    \end{align*}
    Integrating in time from $0$ to $t$ yields
    \begin{align*}
        \|\psi\|^2_{H^{3}}&\lesssim \|\psi(0)\|^2_{H^{3}}+\ep_1^3 \les \ep_0^2+\ep_1^3\lesssim \ep_0^2.
    \end{align*}
    
    Finally, we prove the energy estimate for $Z^a\psi$ with $|a|=1$. By $\psi$-equation in \eqref{sys-VF} and $\div u=\div A=0$ we have
    \begin{align*}
        &\frac{d}{dt}\|Z^a\psi\|_{H^{1}}\\
        &= \sum_{|n|\in\{0,1\};|a|=1}\Big[-\Re\<\d^n Z^a\psi,\sum_{b+c=a}\d^n((Z^bu+2Z^bA)\cdot \nab Z^c\psi) \>+\Re i\< \d^n Z^a\psi,\d^n \mathcal N^{(a)}\>\Big]\\
        &\lesssim  \sum_{|a|=1}\| Z^a\psi\|_{H^{1}}\Big(\|\nab^{\frac 1 2}(Z^au,Z^aA)\|_{H^{1}}\|\nab\psi\|_{W^{1,6}}\\
        &\quad\quad+\|\nab(u,A)\|_{W^{1,\infty}}\|Z^a\psi\|_{H^{1}}+\|\mathcal N^{(1)}\|_{H^{1}}\Big).
    \end{align*}
    Then this combined with Lemma \ref{Dec_A_Cor}, \eqref{Dec_dpsi}, \eqref{Dec_du} and \eqref{Dec_NonL1} to give 
    \begin{align*}
        \frac{d}{dt}\|Z^a\psi\|_{H^{1}}
        &\les \ep_1 \sum_{|a|\leq 1}\Big( \big(\|\nab^{1/2}Z^a u\|_{H^1}+\ep_1^2 \<t\>^{-1/2}\big)\ep_1 \<t\>^{-1}+ \big(\ep_1 \<t\>^{-5/4}+\ep_1^2\<t\>^{-3/2+3\de/4}\big) \ep_1\\
        &\quad + \ep_1^2\<t\>^{-5/4+3\de}+\ep_1 \<t\>^{-3/4}\|\nab Z^a u\|_{H^1}   \Big)\\
        &\les \ep_1\sum_{|a|\leq 1}\Big(   \ep_1^2 \<t\>^{-5/4+3\de}+\ep_1 \<t\>^{-1}\| Z^a u\|_{H^1}^{1/2}\|\nab Z^a u\|_{H^1}^{1/2}+\ep_1 \<t\>^{-3/4}\|\nab Z^a u\|_{H^1}       \Big)\\
        &\les \ep_1\Big(   \ep_1^2 \<t\>^{-5/4+3\de}+\sum_{|a|\leq 1}\ep_1 \<t\>^{-3/4}\|\nab Z^a u\|_{H^1}       \Big)
    \end{align*}
    Integrating in time yields
    \begin{align*}
        \|Z^a\psi\|^2_{H^{1}}&\lesssim \ep_0^2+\int_0^t \ep_1\Big(\ep_1^2\<s\>^{-5/4+3\de}+\sum_{|a|\leq 1}\ep_1 \<t\>^{-3/4}\|\nab Z^a u\|_{H^1} \Big)ds\\
        &\lesssim \ep_0^2+\ep_1^3\lesssim \ep_0^2.
    \end{align*}
    The energy estimate for $Z^a\psi$ follows, and hence we complete the proof of the proposition.
\end{proof}

\bigskip

Next, we prove the $L^2$ weighted bound \eqref{Main_Prop_result2}.
\begin{prop} \label{Prop_Psia}
	With the hypothesis in Proposition \ref{Main_Prop}, for any $t\in[0,T]$ and profile $\Psi=e^{-it\De}\psi$, we have
	\begin{equation*}      
	\lV x\cdot\nab\Psi\rV_{H^{1}}\lesssim \ep_0.
	\end{equation*}
\end{prop}
\begin{proof}
By scaling vector field $S=t\d_t+x\cdot\nab$, the fact that $[e^{it\De},S]=0$ and \eqref{Main_Prop_result1}, we have
	\begin{equation}\label{keyIneq}
	\begin{aligned}
	\lV x\cdot\nab\Psi\rV_{H^{1}}&\lesssim  \lV S\Psi\rV_{H^{1}}+\lV t\d_t\Psi\rV_{H^{1}}\\
	&\lesssim \lV S\psi\rV_{H^{1}}+\lV t\d_t\Psi\rV_{H^{1}}\lesssim \ep_0+t\lV \d_t \Psi\rV_{H^{1}}.
	\end{aligned}
	\end{equation}
    Then it suffices to estimate the last term $\d_t\Psi$ in the right-hand side of \eqref{keyIneq}. By \eqref{Psi-def} and the $\psi$-equation in \eqref{sys-VF}, we have
    \begin{align*}
        \|\d_t\Psi\|_{H^1}&= \|\d_t(e^{-it\De}\psi)\|_{H^1}  =  \|-ie^{-it\De}(i\d_t\psi+\De\psi)\|_{H^1}\leq \|i\d_t\psi+\De\psi\|_{H^1}\\
        &\les \lV u\cdot\nab\psi\rV_{H^{1}}+\lV A\cdot\nab\psi\rV_{H^{1}}+
	\lV \NN^{(0)}\rV_{H^{1}}.
    \end{align*} 
    Thus it suffices to prove
	\begin{gather}	\label{dtPsi_Term2}
	\lV u\cdot\nab\psi\rV_{H^{1}}+\lV A\cdot\nab\psi\rV_{H^{1}}+
	\lV \NN^{(0)}\rV_{H^{1}}\lesssim \ep_1^2\<t\>^{-5/4+3\de}.
	\end{gather}
	
First, we estimate the term $u\cdot\nab\psi$. Using Littlewood-Paley decomposition, we write $u\cdot\nab\psi$ as
\begin{align*}
    u\cdot\nab\psi&=\sum_{k,k_1,k_2} P_k (u_{k_1}\cdot\nab \psi_{k_2})\\
    &=\sum_{k} P_k (\sum_{k_1\leq k-5}u_{k_1}\cdot\nab \sum_{k_2=k+O(1)}\psi_{k_2})+\sum_{k} P_k (\sum_{k-5\leq k_1\leq k+5}u_{k_1}\cdot\nab \sum_{k_2<k+6}\psi_{k_2})\\
    &\quad +\sum_{k} P_k (\sum_{k_1> k+6}u_{k_1}\cdot\nab \sum_{k_2=k_1+O(1)}\psi_{k_2}).
\end{align*}
For each term $P_k (u_{k_1}\cdot\nab \psi_{k_2})$, by \eqref{Psi-def} and integration by parts we have
\begin{align*}
&\FF P_k(u_{k_1}\cdot\nab\psi_{k_2})(\xi)\\
=&\ \varphi_k(\xi)\int e^{-it|\eta|^2} \widehat{u_{k_1}}(\xi-\eta)\cdot i\eta \widehat{\Psi_{k_2}}(\eta)d\eta\\
=&\ (2t)^{-1}\varphi_k(\xi)\int \nab_{\eta}\widehat{u_{k_1}}(\xi-\eta) \widehat{\psi_{k_2}}(\eta)+\widehat{u_{k_1}}(\xi-\eta)e^{-it|\eta|^2}\nab_{\eta}\widehat{\Psi_{k_2}}(\eta)d\eta.
\end{align*}
Then for the low-high interaction, i.e. $k_1\leq k-5,\ k_2=k+O(1)$. By \eqref{Sym_Pre}, Sobolev embedding and interpolation inequality we obtain
	\begin{align*}
	    \|P_k(u_{<k-5}\cdot\nab\psi_{k})\|_{L^2}\lesssim &\ t^{-1}\big( \|\FF^{-1}(\nab_\eta \widehat{u_{<k-5}}(\eta))\|_{L^6}\|\psi_k\|_{L^3}
	     +\|u_{<k-5}\|_{L^3}\|\FF^{-1}(\nab_\eta \widehat{\Psi_{k}}(\eta))\|_{L^6} \big)\\
	    \lesssim &\ t^{-1}\big(\|\FF^{-1}(|\eta|\nab_\eta \widehat{u_{<k-5}}(\eta))\|_{L^2}\|\psi_k\|_{L^\infty}^{1/3}\|\psi_k\|_{L^2}^{2/3}\\
	    &\ +\|u\|_{L^\infty}^{1/3}\| u\|_{L^2}^{2/3}\|\FF^{-1}(|\eta|\nab_\eta \widehat{\Psi_{k}}(\eta))\|_{L^2} \big).
	\end{align*}
Using \eqref{Dec_u}, \eqref{xidxi-u}, \eqref{Dec_psi} and \eqref{xidxi-psi} to bound this by
	\begin{align*}
	    \big(\sum_k 2^{2k^+}\|P_k(u_{<k-5}\cdot\nab\psi_{k})\|_{L^2}^2\big)^{1/2}\lesssim \ep_1^2 \<t\>^{-5/4}.
	\end{align*}
	For the high-low interaction, i.e. $k-5\leq k_1\leq k+5 ,\ k_2<k+6$. By \eqref{Sym_Pre}, Sobolev embedding and interpolation inequality we have
	\begin{align*}
	    &\sum_{k_2<k+6}\|P_k(u_{k}\cdot\nab\psi_{k_2})\|_{L^2}\\
	    \lesssim &\ \sum_{k_2<k+6}t^{-1}\big( \|\FF^{-1}(\nab_\eta \widehat{u_{k}}(\eta))\|_{L^6}\|\psi_{k_2}\|_{L^3}\\
	    &\ \quad+\|u_{k}\|_{L^{\frac{3}{1+3\de}}}\|\FF^{-1}(\nab_\eta \widehat{\Psi_{k_2}}(\eta))\|_{L^{\frac{6}{1-6\de}}} \big)\\
	    \lesssim &\ t^{-1}\big(  \|\FF^{-1}(|\eta|\nab_\eta \widehat{u_{k}}(\eta))\|_{L^2}\|\psi\|_{W^{1,\frac{3}{1+3\de}}}\\
	    &\ +\|u_k\|_{L^\infty}^{(1-6\de)/3}\| u_k\|_{L^2}^{(2+6\de)/3}(\|\FF^{-1}(|\eta|\nab_\eta \widehat{\Psi}(\eta))\|_{H^1}+\|\psi\|_{H^1}) \big).
	\end{align*}
	We use interporlation inequality, \eqref{Dec_psi}, \eqref{xidxi-u}, \eqref{Dec_u} and \eqref{xidxi-psi} to bound this by
	\begin{align*}
	    \Big(\sum_k 2^{2k^+}\|P_k(u_{k+O(1)}\cdot\nab\psi_{<k+6})\|_{L^2}^2\Big)^{1/2}\lesssim \ep_1^2\<t\>^{-\frac{5}{4}+2\de}
	\end{align*}
	For the high-high interaction, i.e. $k_1>k+5,\ k_2=k_1+O(1)$, by Bernstein's inequality and \eqref{Sym_Pre} we have
	\begin{align*}
	    \|P_k(u_{k_1}\cdot\nab\psi_{k_2})\|_{L^2}\lesssim &\ t^{-1}2^{3\de k}\big( \|\FF^{-1}(\nab_\eta \widehat{u_{k_1}}(\eta))\|_{L^6}\|\psi_{k_2}\|_{L^\frac{3}{1+3\de}}\\
	    &\ +\|u_{k_1}\|_{L^\frac{3}{1+3\de}}\|\FF^{-1}(\nab_\eta \widehat{\Psi_{k_2}}(\eta))\|_{L^6} \big)\\
	    \lesssim &\ t^{-1}2^{3\de k}\big(  \|\FF^{-1}(|\eta|\nab_\eta \widehat{u_{k_1}}(\eta))\|_{L^2}\|\psi_{k_2}\|_{L^\infty}^{\frac{1-6\de}{3}}\|\psi_{k_2}\|_{L^2}^{\frac{2+6\de}{3}}\\
	    &\ +\|u\|_{L^\infty}^{\frac{1-6\de}{3}}\| u\|_{L^2}^{\frac{2+6\de}{3}}\|\FF^{-1}(|\eta|\nab_\eta \widehat{\Psi_{k_2}}(\eta))\|_{L^2} \big).
	\end{align*}
	We then use \eqref{Dec_psi} and \eqref{Dec_u} to bound this by
		\begin{align*}
	    &\sum_{k_1\sim k_2>k+O(1)} \|P_k(u_{k_1}\cdot\nab\psi_{k_2})\|_{L^2}\\
	    \lesssim & \ t^{-1}2^{3\de k}2^{-k_2^+}\ep_1\<t\>^{-\frac{1-6\de}{4}}\big(\|\FF^{-1}(|\eta|\nab_\eta \widehat{u_{k_1}}(\eta))\|_{L^2}
	    +\|\FF^{-1}(|\eta|\nab_\eta \widehat{\Psi_{k_2}}(\eta))\|_{L^2}\big) .
	\end{align*}
	By \eqref{Main_Prop_Ass1}, \eqref{xidxi-u} and \eqref{xidxi-psi}, this implies 
	\begin{align*}
	    &\Big[\sum_k 2^{2k^+}(\sum_{k_1\sim k_2>k+O(1)} \|P_k(u_{k_1}\cdot\nab\psi_{k_2})\|_{L^2})^2\Big]^{1/2}\\
	    \lesssim &\ \ep_1\<t\>^{-\frac{5-6\de}{4}}(\|\FF^{-1}(|\eta|\nab_\eta \widehat{u}(\eta))\|_{H^1}+\| u\|_{H^1}+\|\FF^{-1}(|\eta|\nab_\eta \widehat{\Psi}(\eta))\|_{H^1}+\|\psi\|_{H^1})\\
	    \lesssim &\ \ep_1^2\<t\>^{-\frac{5-6\de}{4}}.
	\end{align*}
	Hence, the bound \eqref{dtPsi_Term2} for $u\cdot\nab\psi$ follows.
	
	By \eqref{Dec_d1/2-A} and \eqref{Dec_dpsi} we bound the second term in \eqref{dtPsi_Term2} by
	\begin{align*}
	    \| A\c\nab\psi\|_{H^{1}}\lesssim \| \nab^{1/2}A\|_{H^{1}}\|\nab\psi\|_{W^{1,6}}\lesssim \ep_1^2\<t\>^{-3/2+3\de/2}.
	\end{align*}
	The last term in \eqref{dtPsi_Term2} can be proved using the similar argument to that of \eqref{Dec_NonL} with regularity index $1$. We have
	\begin{align*}
	    \| \mathcal N^{(0)}\|_{H^{1}}\lesssim \ep_1^2\<t\>^{-5/4+3\de}+\ep_1\<t\>^{-3/4}\|\nab u\|_{H^{1}}\lesssim \ep_1^2\<t\>^{-5/4+3\de}.
	\end{align*}
	This completes the proof of the lemma.
\end{proof}

\bigskip

From the above estimate of $\d_t\Psi$ we obtain the scattering. 

\begin{prop}[Scattering]\label{scattering-Lemma}
    With the hypothesis in Proposition \ref{Main_Prop}, for any $t\in[0,T]$, we obtain the scattering
	\begin{equation*}         
	\lim_{t\rightarrow \infty}\lV \psi-e^{it\De}\Psi_\infty\rV_{H^{1}}= 0.
	\end{equation*}
\end{prop}
\begin{proof}
    By the definition of $\Psi$ \eqref{Psi-def} we have
    \begin{align*}
        \Psi(t)=&\ \Psi(0)+\int_0^t \d_s\Psi(s)ds\\
        =&\ \Psi(0)-\int_0^t e^{-is\De} ((u+2A)\c \nab\psi+\NN^{(0)})(s) ds.
    \end{align*}
    By \eqref{dtPsi_Term2} we obtain
    \begin{align*}
        \|\Psi(t_1)-\Psi(t_2)\|_{H^{1}}\lesssim &\  \int_{t_1}^{t_2} \|((u+2A)\c \nab\psi+\NN^{(0)})(s)\|_{H^{1}}ds\\
        \lesssim &\  \int_{t_1}^{t_2}\ep_1^2\<s\>^{-5/4+3\de}ds\rightarrow 0,\quad {\rm as}\ t_1,t_2\rightarrow\infty.
    \end{align*}
    This motivates us to define the function $\Psi_{\infty}$ as
    \[\Psi_{\infty}:=\Psi(0)-\int_0^\infty e^{-is\De} ((u+2A)\c \nab\psi+\NN^{(0)})(s) ds.\]
    Then we have
    \begin{align*}
        \|\psi-e^{it\De}\Psi_\infty\|_{H^{1}}=&\ \|e^{-it\De}\psi-\Psi_\infty\|_{H^{1}}\\
        \leq &\ \int_t^\infty \| ((u+2A)\c \nab\psi+\NN^{(0)})(s)\|_{H^{1}} ds\\
        \lesssim & \ \ep_1^2 \<t\>^{-1/4+3\de},\quad {\rm as}\ t\rightarrow \infty.
    \end{align*}
    This implies that $\psi$ scatters to the linear solution $e^{it\De}\Psi_\infty$ in the lower regularity Sobolev space $H^{1}$.
\end{proof}
\bigskip

In order to prove the Theorem \ref{Main_thm-ori}, we recall the following local existence result and an useful lemma.
\begin{prop}[Local solution, \cite{Hu20}]       \label{prop-local}
	The Cauchy problem \eqref{Sys-Ori} with $(u_0,\phi_0)\in H^k\times H^{k+1}_Q$, for any integer $k\geq [\frac{d}{2}]+2$, admits a unique local solution $(u,\phi)$ satisfying 
	\begin{equation*}
	\lV u\rV_{H^k}+\lV\nab u\rV_{L^2([0,t];H^k)}+\lV \phi\rV_{H^{k+1}_Q}\leq C(k,\lV u_0\rV_{H^k},\lV\phi_0\rV_{H^{k+1}_Q}),
	\end{equation*}
	for any $t\in[0,T]$, where $T=T(\lV u_0\rV_{H^3},\lV\phi_0\rV_{H^4_Q})$.
\end{prop}

In order to construct the Schr\"odinger flow $\phi$ from the solution $\psi$ of \eqref{sys-2}, we need the following lemma.
\begin{lemma}[Bound for $\phi$]  \label{phi-bypsi}
     Assume that $\psi$ is a solution of \eqref{sys-2}-\eqref{A,Ad+1}. If the differentiated fields $\psi$ has the additional property \begin{equation*}
        \sup_{t\in[0,T]}\Big(\|\psi(t)\|_{H^3}+\sum_{|a|=1}\|Z^a\psi(t)\|_{H^{1}}\Big)\leq \ep.
    \end{equation*} 
    Then we have the bound
    \begin{equation*}    
         \sup_{t\in[0,T]}\Big(\|\nab \phi \|_{H^3}+\sum_{|a|=1}\|\nab Z^a \phi \|_{H^{1}}\Big)\lesssim \ep.
     \end{equation*}
\end{lemma}
\begin{proof}
     In fact, we prove
     \begin{equation}     \label{phi-bound}
         \sup_{t\in[0,T]}\Big(\| (\nab \phi,\nab v,\nab w) \|_{H^{3}}+\sum_{|a|=1}\| (\nab Z^a\phi,\nab Z^av,\nab Z^aw) \|_{H^{1}}\Big)\lesssim \ep.
     \end{equation}
     Recall the formulas \eqref{phi-v-w},
     \begin{equation*}   
     \left\{ \begin{aligned}
     &\d_m \phi=v\Re \psi_m+w\Im \psi_m,\\
     &\d_m v=-\phi\Re \psi_m+w A_m,\\
     &\d_m w=-\phi\Im \psi_m- v A_m.
     \end{aligned}
     \right.
     \end{equation*}
     Since $|\phi|=|v|=|w|=1$, by \eqref{A_m-Eq} and Sobolev embedding we have
     \begin{align*}
         \|\nab(\phi,v,w)\|_{L^2}\lesssim \| \psi\|_{L^2}+\|A\|_{L^2}\lesssim \ep.
     \end{align*}
     We then prove the bound the first term in \eqref{phi-bound} by induction. Precisely, assume that 
     \begin{equation*}     
         \|\nab\phi\|_{H^n}+\|\nab v\|_{H^n}+\|\nab w\|_{H^n}\lesssim \ep,\quad \text{for any }n<l\leq 3.
     \end{equation*}
     By Sobolev embedding and the above inductive assumption, we have
     \begin{align*}
         \|\nab^{l+1}\phi\|_{L^2}\lesssim & \ \|\nab^l \psi\|_{L^2}+\sum_{l_1+l_2=l,0<l_2<l}\|\nab^{l_1}\psi\|_{L^4}\|\nab^{l_2}(v,w)\|_{L^4}+\|\psi\|_{L^\infty}\|\nab^l(v,w)\|_{L^2}\\
         \lesssim &\ \|\psi\|_{H^{3}}(1+\|\nab(v,w)\|_{H^{l-1}})\\
         \lesssim &\ \|\psi\|_{H^{3}}\lesssim \ep.
     \end{align*}
     Similarly, by Sobolev embedding and \eqref{A_m-Eq} we also have
     \begin{align*}
         \|\nab^{l+1}v\|_{L^2}+\|\nab^{l+1}w\|_{L^2}
         &\lesssim (\|\psi\|_{H^{3}}+\|A\|_{H^{3}})(1+\|\nab (\phi,v,w)\|_{H^{l-1}})\\
         &\lesssim \|\psi\|_{H^{3}}\lesssim \ep.
     \end{align*}
     
     We continue to bound the second term $\nab Z^a(\phi,v,w)$ for $|a|=1$. By \eqref{phi-v-w} we have
     \begin{align*}
         \sum_{|a|=1}\|\nab Z^a(\phi,v,w) \|_{L^2}&\lesssim \|\nab (\phi,v,w) \|_{L^2}+\sum_{|a|=1}\| Z^a(\psi,A) \|_{L^2}\|(\phi,v,w)\|_{L^\infty}\\
         &\quad +\|(\psi,A)\|_{L^3}\sum_{|a|=1}\|Z^a(\phi,v,w)\|_{L^6}\\
         &\lesssim  \ep+\ep \sum_{|a|=1}\|\nab Z^a(\phi,v,w)\|_{L^2}.
     \end{align*}
     Similarly, we also have
     \begin{align*}
          \sum_{|a|=1}\|\nab^2 Z^a(\phi,v,w) \|_{L^2}&\les \|\nab^2 (\phi,v,w) \|_{L^2}+\sum_{|a|=1}\| Z^a(\psi,A) \|_{H^1}\|(\phi,v,w)\|_{W^{1,\infty}}\\
         &\quad +\|\nab (\psi,A)\|_{L^3}\sum_{|a|=1}\|Z^a(\phi,v,w)\|_{L^6}\\
         &\quad +\| (\psi,A)\|_{L^\infty}\sum_{|a|=1}\|\nab Z^a(\phi,v,w)\|_{L^2}\\
         &\lesssim  \ep.
     \end{align*}
     These conclude the bound \eqref{phi-bound}, and complete the proof of the lemma.
\end{proof}

As a corollary, we can obtain the long time behavior \eqref{thm-scatter}.
\begin{cor}\label{thm-scatter-cor}
    With the hypothesis in Proposition \ref{Main_Prop}, there exists $v_\infty, w_\infty$ and $\Psi_\infty\in H^1$ such that 
	\begin{align}  \label{thm-scatter-re}
	\lim_{t\rightarrow\infty}\|\d_j\phi-v_\infty \Re( e^{it\De}\Psi_{\infty,j})-w_\infty \Im( e^{it\De}\Psi_{\infty,j})\|_{H^1}=0.
	\end{align}
\end{cor}
\begin{proof}
    In Coulomb gauge, the $\psi$ is scattering given by Proposition \ref{scattering-Lemma}. Then for any $1\leq j\leq 3$ we have
     \begin{align}   \label{scatter-re}
     \lim_{t\rightarrow\infty} \|\d_j \phi\cdot v-\Re(e^{it\De}\Psi_{\infty,j})\|_{H^1}+\|\d_j \phi\cdot w-\Im(e^{it\De}\Psi_{\infty,j})\|_{H^1}=0.
     \end{align}
     We choose 
     $$v_{\infty}=\lim_{|x|\rightarrow \infty}v(t,x),\quad w_{\infty}=\lim_{|x|\rightarrow \infty}w(t,x),$$
     which are fixed vectors satisfying $v_\infty, w_\infty \perp Q$. In order to bound the left hand of \eqref{thm-scatter}, we rewrite it as 
\begin{align} \nonumber
    &\d_j \phi-v_\infty \Re(e^{it\De}\Psi_{\infty,j})-w_\infty \Im(e^{it\De}\Psi_{\infty,j})\\ \label{RHS}
    &=\big(\d_j \phi\cdot v-v\cdot v_\infty \Re(e^{it\De}\Psi_{\infty,j})-v\cdot w_\infty \Im(e^{it\De}\Psi_{\infty,j})\big) v\\ \nonumber
    &\quad +\big(\d_j \phi\cdot w-w\cdot v_\infty \Re(e^{it\De}\Psi_{\infty,j})-w\cdot w_\infty \Im(e^{it\De}\Psi_{\infty,j})\big) v.
\end{align}
Here the above two terms in the right hand side are estimated similarly, we only bound the first term \eqref{RHS} in detail.

Indeed, by H\"older inequality and \eqref{phi-bound} we have
\begin{align} \label{cor-scatter}
    &\|\big(\d_j \phi\cdot v-v\cdot v_\infty \Re(e^{it\De}\Psi_{\infty,j})-v\cdot w_\infty \Im(e^{it\De}\Psi_{\infty,j})\big) v\|_{H^1}\\\nonumber
    &\les \|\big(\d_j \phi\cdot v-v\cdot v_\infty \Re(e^{it\De}\Psi_{\infty,j})-v\cdot w_\infty \Im(e^{it\De}\Psi_{\infty,j})\big)\|_{H^1} \|v\|_{W^{1,\infty}}.
\end{align}
Since $\|v\|_{W^{1,\infty}}\leq 1+\ep$ from \eqref{phi-bound}, it suffices to estimate the above first term. Using $|v_{\infty}|=|w_{\infty}=1$ and $v\cdot w=0$, we rewrite it as 
\begin{align*}
    \eqref{cor-scatter}&\les \|\big(\d_j \phi\cdot v-\Re(e^{it\De}\Psi_{\infty,j})\|_{H^1}+\|(v-v_{\infty})\cdot v_\infty \Re(e^{it\De}\Psi_{\infty,j})\|_{H^1}\\
    &\quad +\|v\cdot(w- w_\infty) \Im(e^{it\De}\Psi_{\infty,j})\big)\|_{H^1}. 
\end{align*}
From \eqref{scatter-re}, the first term converges to $0$
\begin{align*}
    \lim_{t\rightarrow\infty}\|\big(\d_j \phi\cdot v-\Re(e^{it\De}\Psi_{\infty,j})\|_{H^1}=0.
\end{align*}
For the second term, by \eqref{v-cong} and $\Psi_{\infty,j}\in H^1$ we have
\begin{align*}
    \lim_{t\rightarrow\infty}\|(v-v_{\infty})\cdot v_\infty \Re(e^{it\De}\Psi_{\infty,j})\|_{H^1}
    \les \lim_{t\rightarrow\infty}\|v-v_\infty\|_{W^{1,\infty}}\|\Re(e^{it\De}\Psi_{\infty,j})\|_{H^1}=0.
\end{align*}
Similarly, by $\|v\|_{W^{1,\infty}}\leq 1$, \eqref{v-cong} and $\Psi_{\infty,j}\in H^1$, we bound the third term by
\begin{align*}
    &\lim_{t\rightarrow\infty}\|v\cdot(w-w_{\infty}) \Re(e^{it\De}\Psi_{\infty,j})\|_{H^1}\\
    &\les \lim_{t\rightarrow\infty}\|v\|_{W^{1,\infty}}\|w-w_\infty\|_{W^{1,\infty}}\|\Re(e^{it\De}\Psi_{\infty,j})\|_{H^1}=0.
\end{align*}
Thus the term \eqref{RHS} converges to $0$ in the space $H^1$ as $t\rightarrow\infty$, i.e.
\begin{align*}
    \lim_{t\rightarrow\infty}\|\big(\d_j \phi\cdot v-v\cdot v_\infty \Re(e^{it\De}\Psi_{\infty,j})-v\cdot w_\infty \Im(e^{it\De}\Psi_{\infty,j})\big) v\|_{H^1}=0.
\end{align*}
For the other term there we have the same convergence. This yields the scattering \eqref{thm-scatter-re}. We complete the proof of corollary.
\end{proof}

\bigskip

Finally, from the above Proposition \ref{prop-local}, the bootstrap Proposition \ref{Main_Prop} and Lemma \ref{phi-bypsi}, we prove our main result:
\begin{proof}[Proof of Theorem \ref{Main_thm-ori}]
     
     \ 
     
     First, We use continuity method to prove global existence. Since initial data $(u_0,\phi_0)$ satisfies \eqref{MainAss_ini}, by Proposition \ref{prop-local} we assume that there exists maximal lifespan $T>1$ such that for any $0\leq t\leq T$
     \begin{equation}  \label{BtAs}
 \|(u,\phi)\|_{H^3\times H^4_Q}+ \lV\nab u\rV_{L^2([0,t]:H^3)}+\sum_{|a|=1} \Big\{\lV (Z^a u,\nab Z^a\phi)\rV_{H^1}+ \lV\nab Z^a u\rV_{L^2([0,t]:H^1)}\Big\}\les \ep_1,
	\end{equation}
	and in the Coulomb gauge
	\begin{equation*}
	    \| x\cdot \nab\Psi\|_{H^{1}}\lesssim \ep_1,
	\end{equation*}
    where $\ep_1=\ep_0^{2/3}$, $\Psi=e^{-it\De}\psi$ and $\psi=\d \phi\cdot v+i\d \phi\cdot w$ as in section \ref{sec2}.
    Then by Lemma \ref{psi-by_phi-Lemma}, initial data \eqref{MainAss_ini} and bootstrap assumption \eqref{BtAs}, we obtain that $(u,\psi)$ is a solution of system \eqref{sys-2} with initial data $(u_0,\psi_0)$ satisfying
    \begin{align*}
        \|(u_0,\psi_0)\|_{H^3}+\sum_{|a|=1}\| (\Lambda^a u_0,\Lambda^a\psi_0)\|_{H^{1}}\lesssim \ep_0,
    \end{align*}
    and also satisfies the bound 
    \begin{align*}
        \sup_{t\in[0,T]}\Big\{&\|(u,\psi)\|_{H^3}+\|\nab u\|_{L^2([0,t]:H^3)}\\
        &+\sum_{|a|=1}\Big(\|(Z^a u,Z^a\psi)\|_{H^1}+\|\nab Z^a u\|_{L^2([0,t]:H^1)}\Big)\Big\}\les\epsilon_1.
    \end{align*}
    From Proposition \ref{Main_Prop}, we obtain that the solution $(u,\psi)$ have the improved bound
    \begin{equation*} 
        \begin{aligned}
        \sup_{t\in[0,T]}\Big\{&\|(u,\psi)\|_{H^3}+\|\nab u\|_{L^2([0,t]:H^3)}\\
        &+\sum_{|a|=1}\Big(\|(Z^a u,Z^a\psi)\|_{H^1}+\|\nab Z^a u\|_{L^2([0,t]:H^1)}\Big)\Big\}\lesssim \epsilon_0.  
    \end{aligned}
    \end{equation*}
    By mass conservation \eqref{mass-Conservation} and Lemma \ref{phi-bypsi}, the above bound gives the improved bound 
\begin{equation}  \label{energy-bound-1}
 \|(u,\phi)\|_{H^3\times H^4_Q}+ \lV\nab u\rV_{L^2([0,t]:H^3)}+\sum_{|a|=1} \Big\{\lV (Z^a u,\nab Z^a\phi)\rV_{H^1}+ \lV\nab Z^a u\rV_{L^2([0,t]:H^1)}\Big\}\les \ep_0,
\end{equation}
    Hence, from continuity method we can extend the solution $(u,\phi)$ by Proposition \ref{prop-local} and obtain the global solution. The global bound \eqref{energy-bound} is also obtained from \eqref{energy-bound-1}.
     
     Next, by Bernstein's inequality we have 
     \begin{align*}
         \|\phi-Q\|_{L^\infty}\lesssim &\ \sum_k 2^{dk/p} \| P_k(\phi-Q)\|_{L^2}^{2/p}\| P_k(\phi-Q)\|_{L^\infty}^{1-2/p}\\
         \lesssim &\  \sum_k 2^{(\frac{d+2}{p}-1)k} \| P_k(\phi-Q)\|_{L^2}^{2/p}\| \nab P_k(\phi-Q)\|_{L^\infty}^{1-2/p}\\
         \lesssim &\  \|\phi-Q\|_{H^1}^{2/p}\| \nab\phi\|_{L^\infty}^{1-2/p},
     \end{align*}
     where we choose $p=d+2-\de$ with $\de>0$ small. Then by \eqref{phi-v-w} and \eqref{Dec_psi} we obtain
     \begin{align*}
         \|\phi-Q\|_{L^\infty}\lesssim \ep_0 \<t\>^{-\frac{3}{4}+\frac{3}{2p}}\rightarrow 0,\quad \text{as }t\rightarrow \infty.
     \end{align*}
     The decay of $u$ is obtained by \eqref{Dec_u}. 
     Thus the bound \eqref{thm-decay} follows. The asymptotic behavior \eqref{thm-scatter} is proved in Corollary \ref{thm-scatter-cor}.  
     Hence we complete the proof of Theorem \ref{Main_thm-ori}.
\end{proof}

\bigskip 
\section*{Acknowledgment}
J. Huang is supported by Beijing Institute of Technology Research Fund Program for Young Scholars. L. Zhao is supported by NSFC Grant of China No. 12271497 and the National Key Research and Development Program of China No. 2020YFA0713100.

\end{document}